\documentclass[onefignum,onetabnum]{siamart190516}
\usepackage{lipsum}
\usepackage{amsfonts}
\usepackage{graphicx}
\usepackage{epstopdf}
\usepackage{booktabs}
\usepackage{mathtools}

\usepackage{algpseudocode}
\algnewcommand{\Initialize}[1]{
  \State \textbf{Initialize:}
  \Statex \hspace*{\algorithmicindent}\parbox[t]{.8\linewidth}{\raggedright #1}
}
\algnewcommand{\Indent}[2]{
  \State {#1}
  \vspace{-2mm}
  \Statex \hspace*{\algorithmicindent}\parbox[t]{.9\linewidth}{\raggedright #2}
}

\ifpdf
  \DeclareGraphicsExtensions{.eps,.pdf,.png,.jpg}
\else
  \DeclareGraphicsExtensions{.eps}
\fi

\usepackage{enumitem}

\usepackage[caption=false]{subfig}
\usepackage{amssymb}

\newcommand{\weak}{\rightharpoonup}
\newcommand{\weakstar}{\stackrel{\ast}{\rightharpoonup}}
\newcommand{\vect}[1]{\boldsymbol{#1}}
\newcommand{\vectt}[1]{\boldsymbol{\mathbf{#1}}}
\newcommand{\dx}{\mathrm{d}x}
\newcommand{\divv}[1]{\mathrm{div}(#1)}
\newcommand{\tens}[1]{\pmb{\mathsf{#1}}}
\newcommand{\Hdiv}{\vect{H}(\mathrm{div}; \Omega)}
\newcommand{\Hgdiv}{\vect{H}_{\vect{g},\mathrm{div}}(\mathrm{div}; \Omega)}
\newcommand{\Hzdiv}{\vect{H}_{0}(\mathrm{div}; \Omega)}
\newcommand{\Hgnodiv}{\vect{H}_{\vect{g}}(\mathrm{div}; \Omega)}
\newcommand{\Hu}{H^1_{\vect{g},\mathrm{div}}(\Omega)^d}
\newcommand{\DG}{\mathrm{DG}}

\newcommand{\lsb}{[\![}
\newcommand{\rsb}{]\!]}
\newcommand{\lcb}{\{\!\!\{}
\newcommand{\rcb}{\}\!\!\}}
\definecolor{deepgreen}{rgb}{0.1, 0.55, 0.1}

\newsiamremark{remark}{Remark}
\newsiamremark{hypothesis}{Hypothesis}
\crefname{hypothesis}{Hypothesis}{Hypotheses}
\newsiamthm{claim}{Claim}

\headers{Analysis of a DG Method for the BP problem}{I.~P.~A.~Papadopoulos}

\title{Numerical analysis of a discontinuous Galerkin method for the Borrvall--Petersson topology optimization problem  \thanks{Submitted DATE.
\funding{The author is supported by the EPSRC Centre for Doctoral Training in Partial Differential Equations: Analysis and Applications [grant number  EP/L015811/1] and The MathWorks, Inc.}}}

\author{Ioannis P.~A.~Papadopoulos\thanks{Mathematical Institute, University of Oxford, Oxford, UK (\email{ioannis.papadopoulos@maths.ox.ac.uk)}.}}

\usepackage{amsopn}

\ifpdf
\hypersetup{
  pdftitle={Analysis of a DG Method for the BP problem},
  pdfauthor={I. P. A. Papadopoulos}
}
\fi

\begin{document}

\maketitle

\begin{abstract}
Divergence-free discontinuous Galerkin (DG) finite element methods offer a suitable discretization for the pointwise divergence-free numerical solution of Borrvall and Petersson's model for the topology optimization of fluids in Stokes flow [Topology optimization of fluids in Stokes flow, International Journal for Numerical Methods in Fluids 41 (1) (2003) 77--107]. The convergence results currently found in literature only consider $H^1$-conforming discretizations for the velocity.  In this work, we extend the numerical analysis of Papadopoulos and S\"uli to divergence-free DG methods with an interior penalty [I.~P.~A.~Papadopoulos and E.~S\"uli, Numerical  analysis  of  a  topology optimization problem for Stokes flow, arXiv preprint arXiv:2102.10408, (2021)]. We show that, given an isolated minimizer of the  infinite-dimensional problem,  there exists a sequence of DG finite element solutions, satisfying necessary first-order optimality conditions, that strongly converges to the minimizer.
\end{abstract}



\section{Introduction}

The Borrvall--Petersson problem is the first model in literature for the topology optimization of fluid flow \cite{Borrvall2003}. Given a restriction on the proportion of a design domain that a Stokes fluid can occupy, the problem attempts to find the optimal channels from the inlets to the outlets that minimize the power dissipation of the flow. No prior assumptions are required for the shape or topology of the solution resulting in a flexible optimization process \cite{Allaire2012, Bendsoe2004}. This flexibility comes at the cost of a nonconvex optimization problem with PDE, box, and inequality constraints that often supports multiple minimizers. 

The topology of the solution is encoded in the material distribution, $\rho : \Omega \to [0,1]$; a function that maps from the coordinates of the domain to the unit interval. Regions where $\{\rho = 1\} \subset \Omega$ are physically interpreted as the channels in which the fluid flows through. Whereas the regions $\{\rho = 0\} \subset \Omega$ are highly impermeable and thus there is negligible fluid flow in these areas. Intermediate values $\{0 < \rho < 1\} \subset \Omega$ are difficult to interpret. However, these regions are penalized in the Borrvall--Petersson model via an inverse permeability function $\alpha$. 

Since Borrvall and Petersson's seminal work, there have been numerous extensions. Evgrafov \cite{Evgrafov2006}, Olesen et al.~\cite{Olesen2005},  and Gersborg-Hansen et al.~\cite{Gersborg2005} extended the model to fluids satisfying the steady-state Navier--Stokes flow. Kreissl et al.~\cite{Kreissl2011} and Deng et al.~\cite{Deng2011} were the first to consider unsteady Navier--Stokes flow and Deng et al.~\cite{Deng2013} later included body forces.  Alonso et al.~extended the model to rotating bodies in cylindrical coordinates \cite{Alonso2018, Alonso2020}. For a detailed review on the literature of the topology optimization of fluids,  we refer to the work of Alexandersen and Andreasen \cite{Alexandersen2020} .

Due to the nonlinear nature of the problem,  solutions of the infinite-dimensional problem  are often difficult to find. Hence, the problem is typically discretized with the finite element method and the minimizers are computed numerically.  The first result concerning finite element convergence can be found in the original paper by Borrvall and Petersson \cite{Borrvall2003}.  They showed that a conforming and inf-sup stable finite element discretization for the velocity-pressure pair together with a piecewise constant discretization for the material distribution weakly(-*) converges to an unspecified minimizer of the problem.  Thore \cite{Thore2021} proved a similar result for a low-order finite element approximation for the velocity-pressure pair that utilizes a penalty on the jumps of the pressure to overcome the violation of the inf-sup stability.  Recently, the  original Borrvall--Petersson result  was improved and extended by Papadopoulos and S\"uli \cite{Papadopoulos2021b} for conforming and inf-sup stable finite element discretizations for the velocity-pressure pair and conforming discretizations for the material distribution. They showed that for \emph{every} isolated minimizer of the problem, there exists a sequence of finite element solutions, to the first-order optimality conditions, that \emph{strongly} converges to the minimizer. Their analysis resolved a number of outstanding issues; namely the lack of checkerboarding in material distribution approximation as the mesh size tends to zero and whether every minimizer of the problem could be well approximated by a finite element method. The difficulty in the analysis is primarily due to the nonconvexity of problem. In order to account for the possibility of multiple minimizers, the authors fix a  minimizer of the infinite-dimensional problem  and construct a modified optimization problem with the fixed  minimizer  as its unique minimizer. Then, strong convergence of the finite element minimizers to the modified optimization problem is proven. The modified optimization problem is then related back to the original optimization problem by showing that a subsequence of the finite element minimizers also satisfy the first-order optimality conditions of the original optimization problem.

The analysis by Borrvall and Petersson \cite{Borrvall2003},  Thore \cite{Thore2021},  and Papadopoulos and S\"uli \cite{Papadopoulos2021b} heavily relied on the use of an $H^1$-conforming finite element discretization for the velocity. In the past couple of decades, discontinuous Galerkin (DG) methods for fluid flow have become increasingly popular \cite{Cockburn2002, Cockburn2007, Gauger2019, Konno2011, Konno2012}. This is in part due to the existence of divergence-free DG finite element methods. Some stable finite element methods for fluid flow, such as the Taylor--Hood finite element pair,  do not satisfy the incompressibility constraint, $\divv{\vect{u}} = 0$, pointwise.  This manifests as a dependence of the error in the velocity on the best approximation error in the pressure. In some problems,  pointwise violation of the incompressibility constraint  has been observed to support instabilities that result in nonphysical solutions \cite{John2017, Linke2019}. In divergence-free finite element methods, the incompressibility constraint is satisfied pointwise which is useful for ensuring pressure robustness \cite{Scott1985} and deriving error bounds on the velocity that are independent of the error of the pressure. 

In Borrvall--Petersson topology optimization problems, a natural mesh refinement to obtain sharper solutions is in regions where $0 < \rho < 1$ a.e. It can be empirically checked that mesh refinement in these regions does little to improve the error in the pressure. If the convergence for the velocity and material distribution rely heavily on the convergence of the pressure, then only doing mesh refinement in those areas caps the improvement in the errors for the velocity and material distribution. This motivates the need for discretizations that decouple the dependence of errors of the velocity and material distribution with the approximation error of the pressure. Divergence-free finite element discretizations also allow for an easier characterization of the kernel of the discretized grad-div term. This characterization has applications in preconditioners for systems arising in incompressible fluid flow \cite{Farrell2021, Farrell2019a, Hong2016, Papadopoulos2021c, Schoberl1999}. 

$H^1$-conforming divergence-free finite element methods exist, for example the \\Scott--Vogelius finite element \cite{Scott1985}. To ensure inf-sup stability for a general mesh in a $k$-th order Scott--Vogelius finite element method, the polynomial order for the velocity space must be $k \geq 2d$ where $d \in \{2,3\}$ is the dimension of the problem \cite{Scott1985,Zhang2011}. The expense of the high order method is normally justified by the accompanying high convergence rate. However, the material distribution $\rho$ is often discretized with piecewise constant or continuous piecewise linear finite elements due to the box constraints on the material distribution. The box constraints not only cause algorithmic restrictions but also reduce the regularity of $\rho$. The relatively low order approximation of the material distribution then caps the order of convergence of the velocity and pressure \cite[Sec.~5.1]{Papadopoulos2021b} which negates the advantage of the high order method. Inf-sup stability can be achieved for $k\geq d$ if the mesh is barycentrically refined \cite{Qin1994}. This was successfully implemented for the double-pipe problem \cite[Sec.~4.1]{Papadopoulos2021a},  an example of the Borrvall--Petersson problem with a local and global minimizer. However, barycentrically refined meshes can be difficult to align with jumps in the  material distribution that solves the infinite-dimensional problem,  which can lead to poorly resolved solutions. Moreover, barycentrically refined meshes complicate the generation of a mesh hierarchy for robust multigrid cycles \cite{Farrell2021}. In contrast, there exist low-order divergence-free DG finite element methods that are inf-sup stable on general meshes. 

In this paper, we extend the results of Papadopoulos and S\"uli \cite{Papadopoulos2021b} to divergence-free DG finite element methods. A standard interior penalty method is used to control the jumps across the facets. Our main result is to show that for every isolated minimizer of the Borrvall--Petersson optimization problem, there exists a sequence of DG finite element solutions to the discretized first-order optimality conditions that strongly converges in the appropriate norms to the minimizer. In particular, if $(\vect{u}, \rho)$ is an isolated velocity-material distribution minimizer to the Borrvall--Petersson problem, and $p$ is the associated pressure, then there exists a sequence of strongly converging finite element solutions $(\vect{u}_h, p_h, \rho_h)$ such that $\| \vect{u} - \vect{u}_h \|_{H^1_{\vect{g}}(\mathcal{T}_h)} \to 0$, $\| p - p_h \|_{L^2(\Omega)} \to 0$, and $\| \rho - \rho_h \|_{L^s(\Omega)} \to 0$, $s \in [1, \infty)$, where $\|\cdot\|_{H^1_{\vect{g}}(\mathcal{T}_h)}$ is the broken $H^1_{\vect{g}}$-norm as defined in \cref{def:H1gbrokennorm}. This analysis ensures that \emph{every} isolated minimizer is well approximated by the divergence-free DG finite element method as the mesh size tends to zero. 

\section{Topology optimization of Stokes flow}
\label{sec:BP}
Given a volume constraint on a fluid in a fixed  bounded  Lipschitz domain $\Omega \subset \mathbb{R}^d$, $d \in \{2, 3\}$, the Borrvall--Petersson model attempts to minimize the energy lost by the flow due to viscous dissipation, whilst maximizing the flow velocities at the applied body force. More precisely, the objective is to find $(\vect{u},\rho) \in \Hu \times C_\gamma$ that minimizes
\begin{align}
J(\vect{u},\rho) \coloneqq\frac{1}{2} \int_\Omega \left(\alpha(\rho) |\vect{u}|^2 + \nu |\nabla \vect{u}|^2 - 2\vect{f} \cdot \vect{u}\right)  \text{d}x,\label{borrvallmin} \tag{BP}
\end{align}
where $\vect{u}: \Omega \to \mathbb{R}^d$ denotes the velocity of the fluid, $\rho: \Omega \to [0,1]$ is the material distribution of the fluid,  $| \cdot |$ denotes the Euclidean norm for $\vect{u}$ and the Frobenius norm for $\nabla \vect{u}$,  and
\begin{align*}
H^1_{\vect{g}}(\Omega)^d &\coloneqq \{\vect{v} \in H^1(\Omega)^d: \vect{v} = \vect{g} \; \text{on} \; \partial \Omega \},\\
\Hu &\coloneqq \{\vect{v} \in H^1_{\vect{g}}(\Omega)^d : \divv{\vect{v}} = 0 \; \text{a.e.\ in} \; \Omega \},\\
C_{\gamma} &\coloneqq \left\{ \eta \in L^\infty(\Omega) : 0 \leq \eta \leq 1 \; \text{a.e.}, \;\; \int_\Omega \eta \; \text{d}x \leq \gamma |\Omega| \right \}.
\end{align*}
Here,  $H^s(\Omega)$, $0 < s < \infty$, and $L^q(\Omega)$, $0 < q \leq \infty$, denote the standard Sobolev ($W^{s,2}(\Omega)$) and Lebesgue spaces, respectively \cite{Adams2003}.  Furthermore, $\vect{f} \in L^2(\Omega)^d$ is a body force, $\nu > 0$ is the (constant) viscosity, and $\gamma \in (0,1)$ is the volume fraction. The equality $\vect{u} = \vect{g}$ on $\partial \Omega$ is to be understood in the boundary trace sense \cite[Sec.~5.5]{Evans2010}. Moreover, the boundary data $\vect{g} \in H^{1/2}(\partial \Omega)^d$ and $\vect{g} = \vect{0} $ on $\Gamma \subset \partial \Omega$, with $\mathcal{H}^{d-1}(\Gamma)>0$, i.e.\ $\Gamma$ has nonzero Hausdorff measure on the boundary. Borrvall and Petersson introduced the inverse permeability term, $\alpha$, which models the influence of the material distribution on the flow. For values of $\rho$ close to one, $\alpha(\rho)$ is small, permitting fluid flow; for values of $\rho$ close to zero, $\alpha(\rho)$ is very large, restricting fluid flow. The function $\alpha$ satisfies the following properties:
\begin{enumerate}[label=({A}\arabic*)]
\item $\alpha: [0,1] \to [\underline{\alpha}, \overline{\alpha}]$ with $0 \leq \underline{\alpha} < \overline{\alpha} < \infty$;
\label{alpha1}
\item $\alpha$ is strongly convex and monotonically decreasing; \label{alpha2}
\item $\alpha(0) = \overline{\alpha}$ and $\alpha(1) = \underline{\alpha}$;
\label{alpha3}
\item $\alpha$ is twice continuously differentiable, \label{alpha4}
\end{enumerate}
generating an operator also denoted $\alpha: C_\gamma \to L^\infty(\Omega; [\underline{\alpha},\overline{\alpha}])$. Typically, in the literature $\alpha$ takes the form \cite{Borrvall2003, Evgrafov2014}
\begin{align}
\alpha(\rho) = \bar{\alpha}\left( 1 - \frac{\rho(q+1)}{\rho+q}\right), \label{eq:alphachoice}
\end{align}
where $q>0$ is a penalty parameter, so that $\lim_{q \to \infty} \alpha(\rho) = \bar{\alpha}(1-\rho)$. Borrvall and Petersson \cite[Sec.~3.2]{Borrvall2003} remark that as $q \to \infty$ the material distribution tends to a 0-1 solution.

We define the following spaces that will be used throughout this work:
\begin{align}
L^2_0(\Omega) &\coloneqq \left\{ q \in L^2(\Omega) : \int_\Omega q \; \dx  = 0 \right \},\\
H^1_0(\Omega)^d &\coloneqq \{\vect{v} \in H^1(\Omega)^d: \vect{v} = \vect{0} \; \text{on} \; \partial \Omega \},\\
\Hdiv &\coloneqq \left\{ \vect{v} \in L^2(\Omega)^d : \divv{\vect{v}} \in L^2(\Omega) \right\},\\
\Hzdiv &\coloneqq \left\{ \vect{v} \in \Hdiv : \vect{v}\cdot \vect{n} = 0 \; \text{on} \; \partial \Omega  \right\},\\
\Hgnodiv &\coloneqq \left\{ \vect{v} \in \Hdiv : (\vect{v} - \vect{g})\cdot \vect{n} = 0 \; \text{on} \; \partial \Omega \right\},\\
\Hgdiv &\coloneqq \left\{ \vect{v} \in \Hgnodiv : \divv{\vect{v}} =0 \; \text{a.e.~in} \; \Omega  \right\}.
\end{align}
We note that $\vect{v}\cdot \vect{n}$ on $\partial \Omega$ is well-defined for all $\vect{v} \in \Hdiv$ \cite[Th.~3.12]{Arnold2018}. Moreover, $\Hdiv$ is a Hilbert space when equipped with the inner product
\begin{align}
(\vect{u}, \vect{v})_{\Hdiv} \coloneqq \int_\Omega \vect{u} \cdot \vect{v} + \divv{\vect{u}}\, \divv{\vect{v}}\, \dx.
\end{align}
We extend the definition of $J$ in \cref{borrvallmin} to functions $\vect{v} \not \in H^1_{\vect{g}}(\Omega)^d$ by
\begin{align}
J(\vect{v}, \eta) = + \infty \;\; \text{for all} \;\; \vect{v} \not\in H^1_{\vect{g}}(\Omega)^d, \, \eta \in C_\gamma.
\label{Hdiv-J}
\end{align}

The following existence theorem is due to Borrvall and Petersson \cite[Th.~3.1]{Borrvall2003}.
\begin{theorem}
\label{th:BPexistence}
Suppose that $\Omega \subset \mathbb{R}^d$ is a Lipschitz domain, with $d\in \{2,3\}$ and $\alpha$ satisfies properties \labelcref{alpha1}--\labelcref{alpha3}.  Then, $J(\vect{v}, \eta)$ is weak$\times$weak-* lower semicontinuous on $H^1(\Omega)^d \times L^\infty(\Omega)$.  Moreover, there exists a pair $(\vect{u}, \rho) \in \Hu \times C_\gamma$ that minimizes $J$ (as defined in \cref{borrvallmin}). 
\end{theorem}
Although we are guaranteed the existence of the minimizer from \cref{th:BPexistence}, the lack of convexity in the functional $J$ means that \cref{borrvallmin} can support multiple minimizers.

We now define the following forms:
\begin{align}
a(\vect{u},\vect{v}; \rho)&\coloneqq\int_\Omega \alpha(\rho) \vect{u} \cdot \vect{v} + \nu \nabla \vect{u} : \nabla \vect{v} \; \dx, \quad 
l(\vect{v}) \coloneqq \int_\Omega \vect{f} \cdot \vect{v} \; \dx, \\
c(\rho, \eta; \vect{u})&\coloneqq\frac{1}{2} \int_\Omega  \alpha'(\rho)|\vect{u}|^2  \eta \; \dx,\\
b(\vect{u},p)&\coloneqq-\int_\Omega  p \, \divv{\vect{u}} \; \dx.
\end{align}

\begin{definition}[Strict minimizer]
Let $Z$ be a Banach space and suppose that $z_0 \in Z$ is a local or global minimizer of the functional $J :Z\to \mathbb{R}$. We say that $z_0$ is a strict minimizer if there exists  an open neighborhood $E \subset Z$ of $z_0$ such that $J(z_0) < J(z)$ for all $z \neq z_0$, $z \in E$.
\end{definition} 
\begin{definition}[Isolated minimizer]
Let $Z$ be a Banach space and suppose that $z_0 \in Z$ is a local or global minimizer of the functional $J :Z\to \mathbb{R}$. We say that $z_0$ is isolated if there exists  an open neighborhood  $E \subset Z$  of $z_0$ such that there are no other minimizers contained in $E$.
\end{definition} 
\begin{remark}
If $z$ is an isolated minimizer, then it is also a strict minimizer. 
\end{remark} 

The proof of the following proposition on first-order optimality conditions for isolated minimizers of \cref{borrvallmin} can be found in Papadopoulos and S\"uli \cite[Prop.~2]{Papadopoulos2021b}.
\begin{proposition}
\label{prop:pressureexistence}
Suppose that $\Omega \subset \mathbb{R}^d$ is a Lipschitz domain, with $d\in\{2,3\}$ and $\alpha$ satisfies properties \ref{alpha1}--\ref{alpha4}.  Consider a minimizer  $(\vect{u}, \rho) \in H^1_{\vect{g}, \mathrm{div}}(\Omega)^d \times C_\gamma$ whose existence is guaranteed by \cref{th:BPexistence}. Then, there exists a unique Lagrange multiplier $p \in L^2_0(\Omega)$ such that the following necessary first-order optimality conditions hold: 
\begin{alignat}{2}
a(\vect{u},\vect{v}; \rho) + b(\vect{v}, p) &= l(\vect{v}) \;\; && \text{for all} \;\; \vect{v} \in H^1_0(\Omega)^d, \tag{FOC1} \label{FOC1}\\
b(\vect{u},q) &=0 \;\; &&  \text{for all} \;\; q \in L^2_0(\Omega), \tag{FOC2} \label{FOC2}\\
c(\rho,\eta-\rho; \vect{u}) &\geq 0 \;\; && \text{for all} \;\; \eta \in C_\gamma.  \tag{FOC3} \label{FOC3}
\end{alignat}
\end{proposition}

\section{Discretization}
\label{sec:discretization}

In this section we fix our choice of discretization. We denote the finite element spaces for the velocity, pressure, and material distribution by $\vect{X}_h \subset \Hdiv$, $M_h \subset L^2_0(\Omega)$ and $C_{\gamma,h} \subset C_\gamma$, respectively. In the following we introduce the notation that is required to define the discretized optimization problem. 

Let $(\mathcal{T}_h)_{h \in (0,1]}$ denote a family of  triangulations of the domain $\Omega \subset \mathbb{R}^d$, $d \in \{2,3\}$, characterized by the mesh size $h\coloneqq\max_{K \in \mathcal{T}_h} h_K$, where $h_K$ is the diameter of the element $K \in \mathcal{T}_h$. We assume that $(\mathcal{T}_h)_{h \in (0,1]}$ satisfies:
\begin{enumerate}[label=({M}\arabic*)]
\item (Shape regularity). There exists constants $c_1, c_2 >0$ such that
\begin{align*}
c_1 h^d_K \leq |K| \leq c_2 h^d_K \;\;\text{for all} \; K \in \mathcal{T}_h.
\end{align*}
\label{ass:M1}
\end{enumerate}
Moreover,  we assume that  $(\mathcal{T}_h)_{h \in (0,1]}$ satisfies a submesh condition as found in Buffa and Ortner \cite[Assumption 2.1]{Buffa2009}. For a given $h$, let the set $\mathcal{F}_h$ denote the set of all facets of the triangulation $\mathcal{T}_h$ and $h_F$ represent the diameter of each facet $F \in \mathcal{F}_h$. We split the set of facets into the union $\mathcal{F}_h = \mathcal{F}^i_h \cup \mathcal{F}^\partial_h$ where $\mathcal{F}^i_h$ is the subset of interior facets and $\mathcal{F}^\partial_h$ collects all Dirichlet boundary facets $F \subset \partial \Omega$.  We note that we only consider Dirichlet boundary conditions in this work, and, hence, all the facets on the boundary are elements of the set $\mathcal{F}^\partial_h$.  We assume the following:
\begin{enumerate}[label=({M}\arabic*)]
\setcounter{enumi}{1}
\item (Contact regularity). There exists a constant $c_1>0$ such that 
\begin{align*}
c_1 h^{d-1}_K \leq \mathcal{H}^{d-1}(F) \;\;\text{for all} \;  F \in \mathcal{F}_h,\, K \in \mathcal{T}_h \;\; \text{such that}\;\; F\subset \bar K;
\end{align*}
\label{ass:M2}
\item (Boundary regularity). There exists a constant $c_1>0$ such that:
\begin{align*}
c_1 h \leq h_F \;\;\text{for all} \;  F \in \mathcal{F}^\partial_h.
\end{align*}
\label{ass:M3}
\end{enumerate}
If $F \in \mathcal{F}^i_h$, then $F= \overline{\partial K^+} \cap \overline{\partial K^-}$ for two elements $K^-, K^+ \in \mathcal{T}_h$. We write $\vect{n}^{+}_F$ and $\vect{n}^{-}_F$ to denote the outward normal unit vectors to the boundaries $\partial K^+$ and $\partial K^-$, respectively. If $F \in \mathcal{F}^\partial_h$, then $\vect{n}_F$ is the outer unit normal vector $\vect{n}$. We denote the space of discontinuous finite element functions with degree no higher than $k$ by 
\begin{align}
X_{\DG_k} \coloneqq\{v \in L^1(\Omega) : v|_K \in \mathcal{P}_k \; \text{for all} \; K \in \mathcal{T}_h \},
\end{align}
where $\mathcal{P}_k$ denotes the set of polynomials of order no higher than $k$. Let $\vect{\phi} \in (X_{\DG_k})^d$ and $\tens{\Phi} \in (X_{\DG_k})^{d \times d}$ be any piecewise vector or matrix-valued function, respectively with traces from within the interior of $K^\pm$ denoted by $\vect{\phi}^\pm$ and $\tens{\Phi}^\pm$, respectively. We define the jump $\lsb \cdot \rsb_F$ and the average $\lcb \cdot \rcb_F$ operators across interior facets $F\in \mathcal{F}^i_h$ by
\begin{align}
\lsb \vect{\phi} \rsb_F = \vect{\phi}^+ \otimes \vect{n}_F^+ + \vect{\phi}^- \otimes \vect{n}_F^- \quad \text{and} \quad \lcb \tens{\Phi} \rcb_F = \frac{1}{2}\left(\tens{\Phi}^+ + \tens{\Phi}^-\right).
\end{align}
 Here, $\otimes$ denotes the tensor product, i.e.~for two vectors $\vectt{a}, \vectt{b} \in \mathbb{R}^d$, $\vectt{a} \otimes \vectt{b} \in \mathbb{R}^{d \times d}$, $[\vectt{a} \otimes \vectt{b}]_{ij} = \vectt{a}_i \vectt{b}_j$.  If $F \in \mathcal{F}^\partial_h$, we set $\lsb \vect{\phi} \rsb_F = \vect{\phi} \otimes \vect{n}_F$ and $\lcb \tens{\Phi} \rcb_F = \tens{\Phi}$.  For any $F \in \mathcal{F}^\partial_h$, we define $\int_F |\lsb \vect{\phi} \rsb_F|^2 \, \mathrm{d}s = \int_F | \vect{\phi} |^2 \, \mathrm{d}s$. 

 In general functions $v \in X_{\DG_k}$ do not live in $H^1(\Omega)$ due to the jumps across the facets of the elements. However, on each element, $K \in \mathcal{T}_h$, $v|_K$ is a polynomial and, therefore, $v \in H^1(K)$.  We define the broken Sobolev space $H^1(\mathcal{T}_h)$ as:
\begin{align}
H^1(\mathcal{T}_h) \coloneqq \{ v \in L^1(\Omega) : v \in H^1(K) \; \text{for all} \; K \in \mathcal{T}_h\}.
\end{align}
Moreover, for a function $\vect{v} \in H^1(\mathcal{T}_h)^d$, we define the broken $H^1$-seminorm and norms as:
\begin{align}
| \vect{v} |^2_{H^1(\mathcal{T}_h)} &\coloneqq \sum_{K \in \mathcal{T}_h} \| \nabla \vect{v} \|^2_{L^2(K)} + \sum_{F \in \mathcal{F}^i_h} \int_F h_F^{-1} | \lsb \vect{v} \rsb_F|^2 \mathrm{d}s, \label{def:H1brokenSeminorm}\\
\| \vect{v} \|^2_{H^1(\mathcal{T}_h)} &\coloneqq \|\vect{v}\|^2_{L^2(\Omega)} + | \vect{v} |^2_{H^1(\mathcal{T}_h)},
\label{def:H1brokennorm}\\
\| \vect{v} \|^2_{H^1_{\vect{g}}(\mathcal{T}_h)} &\coloneqq \| \vect{v} \|^2_{H^1(\mathcal{T}_h)} + \sum_{F \in \mathcal{F}^\partial_h} \int_F h_F^{-1} | \lsb \vect{v} - \vect{g} \rsb_F|^2 \mathrm{d}s.
\label{def:H1gbrokennorm}
\end{align}

The two families of DG finite elements of interest for the velocity are the Brezzi--Douglas--Marini (BDM) finite element \cite{Brezzi1987, Brezzi1985} and the Raviart--Thomas (RT) finite element \cite{Raviart1977, Nedelec1980}. The $k$-th order BDM finite element is defined for $d=2$ in \cite[Sec.~2]{Brezzi1985} and for $d=3$ in \cite[Sec.~2]{Brezzi1987}. Similarly the $k$-th order RT finite element is defined in \cite[Sec.~3]{Raviart1977} and \cite[Sec.~2]{Nedelec1980} for $d=2$ and $d=3$, respectively. The finite element spaces induced by the $k$-th order BDM and RT finite elements are denoted by $\vect{X}_{\mathrm{BDM}_k}$ and $\vect{X}_{\mathrm{RT}_k}$, respectively. We note that $\vect{X}_{\mathrm{RT}_k} \subset \vect{X}_{\mathrm{BDM}_k} \subset \vect{Z}_h$ where, for a given $k \geq 1$,
\begin{align}
\vect{Z}_{h} &\coloneqq \{ \vect{v} \in (X_{\DG_k})^d : \, \divv{\vect{v}} \in X_{\DG_{k-1}} \cap L^2(\Omega)\}.
\end{align}
We note that $\vect{Z}_h \subset \Hdiv$ and $\vect{Z}_h \subset H^1(\mathcal{T}_h)^d$. 

We define the following subspaces of $\vect{X}_h$:
\begin{align}
\vect{X}_{h, 0} &\coloneqq \{ \vect{v} \in \vect{X}_{h} :\vect{v}\cdot \vect{n} = 0 \; \text{on} \; \partial \Omega \},\\
\vect{X}_{h, \vect{g}} &\coloneqq \{ \vect{v} \in \vect{X}_{h} :(\vect{v} - \vect{g}) \cdot \vect{n} = 0 \; \text{on} \; \partial \Omega \}.
\end{align}
In general the boundary data $\vect{g}$ cannot be represented in the finite element space. Hence we instead approximate the boundary data with a finite element function $\vect{g}_h$ (which can be represented) and assume that
\begin{enumerate}[label=({F}\arabic*)]
\item $h^{-1} \|\vect{g} - \vect{g}_h\|_{L^2(\partial \Omega)} \to 0$ as $h \to 0$. 
\label{ass:gh}
\end{enumerate}
We also assume that:
\begin{enumerate}[label=({F}\arabic*)]
\itemsep=0pt
\setcounter{enumi}{1}
\item $\vect{X}_{h,0}$ and $M_h$ satisfy the following inf-sup condition for some $c_b > 0$,  independent of $h$, 
\begin{align}
c_b \leq \inf_{q_h \in M_h \backslash \{0\}} \sup_{\vect{v}_h \in \vect{X}_{h,0} \backslash \{0\}} \frac{b(\vect{v}_h,q_h)}{\|\vect{v}_h\|_{H^1(\mathcal{T}_h)}\|q_h\|_{L^2(\Omega)}}.
\end{align}
\label{ass:discreteinfsup}
\item The finite element spaces are dense in their respective function spaces, i.e., for any $(\vect{v}, \eta, q) \in H^1(\mathcal{T}_h)^d \times C_\gamma \times L^2_0(\Omega)$, \label{ass:dense}
\begin{align*}
\lim_{h \to 0} \inf_{\vect{w}_h \in \vect{X}_{h}}\|\vect{v}-\vect{w}_h\|_{H^1(\mathcal{T}_h)} &= \lim_{h\to 0} \inf_{\zeta_h \in C_{\gamma,h}} \|\eta - \zeta_h \|_{L^2(\Omega)} \\
&\indent = \lim_{h\to 0} \inf_{r_h \in M_h} \|q - r_h \|_{L^2(\Omega)} = 0.
\end{align*}
\end{enumerate}
\begin{remark}
The inf-sup \labelcref{ass:discreteinfsup} and density \labelcref{ass:dense} conditions are satisfied by either $\vect{X}_h = \vect{X}_{\mathrm{BDM}_k}$ or $\vect{X}_h = \vect{X}_{\mathrm{RT}_k}$ with the choice of the pressure finite element space $M_h = X_{\DG_{k-1}}$ \cite{Cockburn2007}. 
\end{remark}

\begin{remark}
Although $C_\gamma$ is not separable, we will only require the density of $C_{\gamma, h}$ in $C_\gamma$ with respect to the $L^2$-norm. 
\end{remark}

We now define the discrete Borrvall--Petersson power dissipation functional for a DG finite element discretization. Consider the functions $\vect{u}_h \in \vect{X}_h \subset \vect{Z}_h$ and $\rho_h \in C_{\gamma,h}$. We note that $J(\vect{u}_h, \rho_h)$ is ill-defined  as, in general, $\nabla \vect{u}_h \not\in L^2(\Omega)^d$. Hence, the term $\int_\Omega |\nabla \vect{u}_h|^2\,\dx$, as found in $J(\vect{u}_h,\rho_h)$, might not be finite.  Given a penalization parameter $\sigma > 0$, we define the discrete analogue $J_h$ as
\begin{align}
\begin{split}
J_h(\vect{u}_h, \rho_h)&\coloneqq \frac{1}{2} \int_\Omega \left(\alpha(\rho_h) |\vect{u}_h|^2 - 2\vect{f} \cdot \vect{u}_h\right)  \text{d}x
+ \frac{\nu}{2}  \sum_{K \in \mathcal{T}_h}\int_K  |\nabla \vect{u}_h|^2 \dx \\
& +   \frac{\nu}{2}  \sum_{F \in \mathcal{F}^i_h} \sigma h_F^{-1} \int_F  |\lsb \vect{u}_h\rsb_F|^2 \mathrm{d}s
 -   \nu  \sum_{F \in \mathcal{F}^i_h}\int_F \lcb \nabla \vect{u}_h \rcb_F : \lsb \vect{u}_h \rsb_F\, \mathrm{d}s\\
&+  \frac{\nu}{2}  \sum_{F \in \mathcal{F}^\partial_h} \sigma h_F^{-1} \int_F  |\lsb \vect{u}_h - \vect{g}_h \rsb_F|^2 \mathrm{d}s\\
& -  \nu  \sum_{F \in \mathcal{F}^\partial_h}\int_F \lcb \nabla \vect{u}_h \rcb_F : \lsb \vect{u}_h  - \vect{g}_h \rsb_F\, \mathrm{d}s.
\end{split}
\end{align}
\begin{remark}
This particular choice $J_h$ as the discrete analogue of $J$ is motivated by an interior penalty approach for DG formulations. In \cref{prop:FEMconvergence3} we prove that velocity minimizers  of $J_h$ satisfy a fluid momentum equation featuring terms that arise in the interior penalty DG discretization of the Stokes equations \cite{Cockburn2002, Gauger2019}.
\end{remark}
\begin{remark}
For any $\vect{v}_h \in \vect{Z}_h$ the terms $\lcb \nabla \vect{v}_h \rcb_F$ and $\lsb \vect{v}_h \rsb_F$ as they appear in $J_h$ are well-defined \cite[Sec.~3.1]{Arnold2002}.
\end{remark}
\begin{proposition}[Consistency of $J_h$]
\label{prop:consistencyJh}
Consider any $(\vect{v}, \eta) \in H^1_{\vect{g}}(\Omega)^d \times C_\gamma$ such that $\vect{v} \in H^r(\Omega)^d$ for some $r>3/2$. Then, $J_h$, $h > 0$ is consistent, i.e.
\begin{align}
J_h(\vect{v}, \eta) = J(\vect{v}, \eta).
\end{align}
\end{proposition}
\begin{proof}
Since $\vect{v} \in H^r(\Omega)^d$, for $r > 3/2$, we note that $J_h(\vect{v}, \eta)$ is well-defined and there can be no jumps in $\vect{v}$ across elements. Hence, for all $F \in \mathcal{F}^i_h$, integrals involving $\lsb \vect{v} \rsb_F$ are equal to zero. Moreover $\vect{v} = \vect{g}$ on $\partial \Omega$ and, therefore, $\int_F |\lsb \vect{v}  - \vect{g} \rsb_F|^2\, \mathrm{d}s = 0$ for all $F \in \mathcal{F}^\partial_h$.  Hence,
\begin{align}
J_h(\vect{v}, \eta) = \frac{1}{2} \int_\Omega \left(\alpha(\eta) |\vect{v}|^2 - 2\vect{f} \cdot \vect{v}\right)  \text{d}x + \frac{\nu}{2}  \sum_{K \in \mathcal{T}_h}\int_K  |\nabla \vect{v}|^2 \dx = J(\vect{v},\eta).
\end{align}
\end{proof}

For a sufficiently large penalization parameter $\sigma > 0$, we define the broken form $a_h(\vect{u}, \vect{v}; \rho)$ by
\begin{align}
\begin{split}
a_h(\vect{u}, \vect{v}; \rho)&\coloneqq \sum_{K\in \mathcal{T}_h} \int_K  \alpha(\rho) \vect{u} \cdot \vect{v}+ \nabla \vect{u} : \nabla \vect{v} \; \dx +  \nu  \sum_{F \in \mathcal{F}_h} \sigma h_F^{-1} \int_F  \lsb \vect{u}\rsb_F : \lsb \vect{v}\rsb_F \mathrm{d}s\\
& -  \nu  \sum_{F \in \mathcal{F}_h}\int_F \lcb \nabla \vect{u} \rcb_F  : \lsb \vect{v} \rsb_F \mathrm{d}s -  \nu  \sum_{F \in \mathcal{F}_h} \int_F  \lsb \vect{u} \rsb_F : \lcb \nabla \vect{v} \rcb_F \mathrm{d}s,
\end{split}
\end{align}
and the linear functional $l_h(\vect{v}; \vect{g})$ as
\begin{align}
\begin{split}
l_h(\vect{v}; \vect{g})&\coloneqq  \int_\Omega \vect{f} \cdot \vect{v} \; \dx \\
& \indent +  \nu  \sum_{F \in \mathcal{F}^\partial_h} \sigma h_F^{-1} \int_F \lsb \vect{g} \rsb_F : \lsb \vect{v} \rsb_F \; \mathrm{d}s -  \nu  \sum_{F \in \mathcal{F}^\partial_h} \int_F\lsb \vect{g} \rsb_F: \lcb \nabla \vect{v} \rcb_F \,\mathrm{d}s.
\end{split}
\end{align}

\begin{proposition}[Consistency of $a_h$]
\label{prop:consistencyah}
Suppose that $(\vect{u}, \rho) \in \Hu \times C_\gamma$ is a minimizer of \cref{borrvallmin} and let $p \in L^2_0(\Omega)$ denote the Lagrange multiplier such that $(\vect{u}, \rho, p)$ satisfy \cref{FOC1}--\cref{FOC3}. Moreover, assume that $\vect{u} \in H^r(\Omega)^d$ for some $r>3/2$. Then, for all $\vect{v}_h  \in H^1(\mathcal{T}_h)^d \cap \Hzdiv$, we have that
\begin{align}
a_h(\vect{u}, \vect{v}_h; \rho) + b(\vect{v}_h, p) = l_h(\vect{v}_h; \vect{g}). \label{eq:consistency}
\end{align}
\end{proposition}
\begin{proof}
By \cref{FOC1}, we have that, for all  $\vect{\phi} \in H^1_0(\Omega)^d$, and therefore, for all $\vect{\phi} \in C_c^\infty(\Omega)^d$,
\begin{align}
a(\vect{u}, \vect{\phi}; \rho) + b(\vect{\phi}, p) = l(\vect{\phi}), 
\end{align}
where $C_c^\infty(\Omega)$ denotes set of smooth and compactly supported functions in $\Omega$. By an integration by parts, we see that
\begin{align}
\int_\Omega \alpha(\rho) \vect{u} \cdot \vect{\phi} - \nu \Delta \vect{u} \cdot \vect{\phi} + \nabla p \cdot \vect{\phi} \, \dx = \int_\Omega \vect{f} \cdot \vect{\phi} \, \dx. \label{eq:distributional}
\end{align}
We note that $\int_\Omega \vect{\psi} \cdot \vect{\phi} \, \dx$ is well-defined for any $\vect{\psi} \in (C_c^\infty(\Omega)^d)^*$ (the space of distributions). Hence, \cref{eq:distributional} is well-defined since $\alpha(\rho) \vect{u}$, $\Delta \vect{u}$, $\nabla p$, $\vect{f} \in (C_c^\infty(\Omega)^d)^*$. 

As the set of smooth functions is dense in $L^2(\Omega)$, we can test \cref{eq:distributional} against any $\vect{v}_h  \in H^1(\mathcal{T}_h)^d \cap \Hzdiv \subset L^2(\Omega)^d$. Thus, by performing a second integration by parts, we have that
\begin{align}
\begin{split}
&\sum_{K \in \mathcal{T}_h} \int_K \alpha(\rho)\vect{u} \cdot \vect{v}_h + \nu \nabla \vect{u} : \nabla \vect{v}_h -p \, \divv{\vect{v}_h} \, \dx  \\
&\indent -  \nu  \sum_{F \in \mathcal{F}_h} \int_{F} \lcb \nabla \vect{u}\rcb_F : \lsb \vect{v}_h \rsb_F \, \mathrm{d}s = \int_\Omega \vect{f} \cdot \vect{v}_h \, \dx.
\end{split}
\end{align}
The element-wise surface integrals arising by the integration by parts of the $\nabla p$ term drop out due to the continuity of $\vect{v}_h \cdot \vect{n}$ across elements for all $\vect{v}_h \in \Hzdiv$. Similarly the boundary surface integrals drop out since $\vect{v}_h \cdot \vect{n} = 0$ on $\partial \Omega$. As $\vect{u} \in \Hu$, for all $F \in \mathcal{F}^i_h$, we have that $\lsb \vect{u} \rsb_F = 0$ and for all $F \in \mathcal{F}^\partial_h$, $\lsb \vect{u} \rsb_F = \lsb \vect{g} \rsb_F$. As $\vect{u} \in H^r(\Omega)^d$, for some $r >3/2$, the traces of $\nabla \vect{u}$ on $F \in \mathcal{F}_h$ are well-defined. We conclude that \cref{eq:consistency} holds. 
\end{proof}

\begin{proposition}[Coercivity and boundedness of $a_h$]
\label{prop:coercivity}
There exists a $\sigma_0 > 0$, such that for all $\sigma \geq \sigma_0$, $\vect{w}_h, \vect{u}_h \in \vect{X}_h$ and $\eta \in C_\gamma$, there exists constants $c_a, C_a>0$ such that
\begin{align}
c_a \|\vect{w}_h \|^2_{H^1(\mathcal{T}_h)} &\leq a_h(\vect{w}_h, \vect{w}_h; \eta),\\
a_h(\vect{w}_h, \vect{u}_h; \eta)&\leq C_a \|\vect{w}_h \|_{H^1(\mathcal{T}_h)} \|\vect{u}_h \|_{H^1(\mathcal{T}_h)}.
\end{align}
\end{proposition}
\begin{proof}
We note that, by assumption \labelcref{alpha1}, $0 \leq \alpha(\eta) \leq \bar \alpha$ for all $\eta \in C_\gamma$. Hence, the result follows from classical coercivity and boundedness results for DG discretizations for interior penalty methods \cite[Sec.~4.1--4.2]{Arnold2002}. 
\end{proof}

\begin{definition}
We define the spaces $\vect{U}_{h, \vect{g}_h}$ and $\vect{U}_{h, 0}$ as:
\begin{align}
\vect{U}_{h, \vect{g}_h} &\coloneqq \{ \vect{u} \in \vect{X}_{h, \vect{g}_h} : b(\vect{u}_h, q_h) = 0 \; \text{for all} \; q_h \in X_{\DG_{k-1}}\}\\
\vect{U}_{h, 0} &\coloneqq \{ \vect{u} \in \vect{X}_{h, 0} : b(\vect{u}_h, q_h) = 0 \; \text{for all} \; q_h \in X_{\DG_{k-1}}\}.
\end{align}
\end{definition}

In the following lemma we provide the proof that functions $\vect{v}_h \in \vect{U}_{h,0}$ and $\vect{v}_h \in \vect{U}_{\vect{g}_h,0}$ are pointwise divergence-free.
\begin{lemma}[Pointwise divergence-free]
\label{lem:pointwisedivfree}
Suppose that $\vect{X}_h \subset \vect{Z}_h$. Consider a function $\vect{v}_h \in \vect{U}_{h,0}$ or $\vect{v}_h \in \vect{U}_{h,\vect{g}_h}$. Then, $\divv{\vect{v}_h} = 0$ a.e.~in $\Omega$.
\end{lemma}
\begin{proof}
Since $\vect{U}_{h,0}, \vect{U}_{h,\vect{g}_h} \subset \vect{X}_h \subset \vect{Z}_h$ then, by definition, $\divv{\vect{v}_h} \in X_{\DG_{k-1}}$. Hence, there exists a $q_h \in X_{\DG_{k-1}}$ such that $q_h = \divv{\vect{v}_h}$. Therefore,
\begin{align}
b(\vect{v}_h, q_h) =  -  \|\divv{\vect{v}_h}\|^2_{L^2(\Omega)} = 0,
\end{align}
which implies that $\divv{\vect{v}_h} = 0$ a.e.~in $\Omega$.
\end{proof}

To prove the convergence of a DG finite element method, we require the existence of sequences in $\vect{U}_{h, \vect{g}_h}$ that converge strongly to $\vect{u}$. 
\begin{lemma}[Strongly converging sequences]
\label{lemma:strongv}
Suppose that \labelcref{ass:gh}--\labelcref{ass:dense} hold and $\vect{X}_h \subset \vect{Z}_h$. Consider any minimizer $(\vect{u},\rho) \in \Hu \times C_\gamma$ of \cref{borrvallmin}. Then, there exists a sequence of functions $(\vect{\tilde{u}}_h, \tilde{p}_h) \in \vect{U}_{\vect{g}_h,h}\times M_h$ such that $\|\vect{u}-\vect{\tilde{u}}_h\|_{H^1_{\vect{g}}(\mathcal{T}_h)} \to 0$ and $\|p - \tilde{p}_h\|_{L^2(\Omega)} \to 0$.
\end{lemma}
\begin{proof}
For sufficiently large $\sigma > 0$ and fixed $\rho \in C_\gamma$, consider the problem, find $(\vect{\tilde{u}}_h, \tilde{p}_h) \in \vect{U}_{\vect{g}_h,h} \times M_h$ that satisfies
\begin{align}
a_h(\vect{\tilde{u}}_{h},\vect{v}_{h}; \rho) + b(\vect{v}_{h}, \tilde{p}_{h}) &= l_h(\vect{v}_{h}) && \text{for all} \; \vect{v}_{h} \in \vect{X}_{h,0},\\
b(\vect{\tilde{u}}_{h},q_{h}) &=0 && \text{for all} \; q_{h} \in M_{h}.
\end{align}
Then, under assumptions \labelcref{ass:gh}--\labelcref{ass:dense}, by standard results for $\Hdiv$ finite element discretizations of the Stokes and Stokes--Brinkman equations with an interior penalty \cite{Cockburn2002, Konno2011, Konno2012}, the pair $(\vect{\tilde{u}}_h, \tilde{p}_h)$ exists, is unique, and $\|\vect{u}-\vect{\tilde{u}}_h\|_{H^1_{\vect{g}}(\mathcal{T}_h)} \to 0$, $\|p - \tilde{p}_h\|_{L^2(\Omega)}\to 0$ as $h\to0$. 
\end{proof}

\section{Convergence of a DG finite element method}
In their original paper Borrvall and Petersson \cite[Sec.~3.3]{Borrvall2003} considered a piecewise constant finite element approximation of the material distribution coupled with an inf-sup stable quadrilateral finite element approximation of the velocity and the pressure. In particular their velocity finite element space was $H^1$-conforming. They showed that if the domain $\Omega$ is a polygonal domain in two dimensions or a polyhedral Lipschitz domain in three dimensions, such approximations of the velocity and material distribution $(\vect{u}_h, \rho_h)$ that minimize $J(\vect{u}_h, \rho_h)$ converge to an unspecified solution $(\vect{u}, \rho)$ of \cref{borrvallmin} in the following sense \cite[Th.~3.2]{Borrvall2003}:
\begin{align*}
\vect{u}_h &\weak \vect{u} \;\; \text{weakly in} \;\; H^1(\Omega)^d, \\
\rho_h &\weakstar \rho \;\;  \text{weakly-* in} \;\; L^\infty(\Omega),\\
\rho_h &\to \rho \;\;  \text{strongly in} \;\; L^s(\Omega_b), \;\; s \in [1,\infty),
\end{align*}
where $\Omega_b$ is any measurable subset of $\Omega$ where $\rho$ is equal to zero or one a.e. Their analysis suggests that a finite element method is a suitable discretization, but it left a number of open problems:
\begin{enumerate}[label=({P}\arabic*)]
\item It is not clear which minimizer the sequence is converging to as the nonconvexity of the problem provides multiple candidates for the limits; \label{open:nonconvexity}
\item The convergence is weak-* in the material distribution in regions where $\{0 < \rho < 1\} \subset \Omega$ which permits the presence of checkerboard patterns as $h \to 0$;
\item There are no convergence results for the finite element approximation of the pressure, $p$. \label{open:pressure}
\end{enumerate}
In general \labelcref{open:nonconvexity} means that their result does not imply that there exists a sequence of finite element solutions that converges to the global minimizer.

Recently Borrvall and Petersson's result was extended and refined by Papadopoulos and S\"uli \cite[Th.~4]{Papadopoulos2021b}. They considered conforming discretizations of the material distribution and conforming inf-sup stable finite element discretizations for the velocity-pressure pair. Once again, the velocity finite element space $\vect{X}_{h, \vect{g}_h}$ was assumed to be $H^1$-conforming. They showed that, for any isolated minimizer $(\vect{u},\rho)$ of \cref{borrvallmin}, there exists a sequence of solutions $(\vect{u}_h, \rho_h, p_h)  \in \vect{X}_{h, \vect{g}_h} \times C_{\gamma,h} \times M_h$ satisfying the discretized first-order optimality condititions, for all $(\vect{v}_h, \eta_h, q_h) \in \vect{X}_{h, 0} \times C_{\gamma,h} \times M_h$:
\begin{align}
a(\vect{u}_h, \vect{v}_h;\rho_h) +  b(\vect{v}_h,p_h)&=l(\vect{v}_h),\\
b(\vect{u}_h,q_h) &=0,\\
c(\rho_h, \eta_h - \rho_h; \vect{u}_h) &\geq 0, 
\end{align}
such that
\begin{align*}
\vect{u}_h &\to \vect{u} \;\; \text{strongly in} \;\; H^1(\Omega)^d, \\
\rho_h &\to\rho \;\;  \text{strongly in} \;\; L^s(\Omega), \; s \in [1,\infty),\\
p_h & \to p \;\; \text{strongly in} \;\; L^2(\Omega).
\end{align*}
Their analysis resolved the open issues \labelcref{open:nonconvexity}--\labelcref{open:pressure}. The assumption that $\vect{X}_{h, \vect{g}_h} \subset H^1(\Omega)^d$ was crucial for the compactness results utilized in their proof. 

Our goal in this section is to prove a similar result for when $\vect{X}_{h, \vect{g}_h} \not\subset H^1(\Omega)^d$ but $\vect{X}_{h, \vect{g}_h} \subset \vect{Z}_h \subset\Hdiv$. The following theorem is the main result of this work. 

\begin{theorem}[Convergence of the DG finite element method]
\label{th:FEMexistence}
Let $\Omega \subset \mathbb{R}^d$ be a polygonal domain in two dimensions or a polyhedral  Lipschitz domain in three dimensions. Suppose that the inverse permeability $\alpha$ satisfies \labelcref{alpha1}--\labelcref{alpha4} and there exists an isolated minimizer $(\vect{u},\rho) \in \Hu \times C_\gamma$ of \cref{borrvallmin} that has the additional regularity $\vect{u} \in H^r(\Omega)^d$ for some $r>3/2$. Moreover, assume that, for $\theta > 0$, $U_\theta$ is the subset of $\Omega$ where $|\vect{u}|^2 \geq \theta$ a.e.~in $U_\theta$ and suppose that there exists a $\theta'>0$ such that $U_\theta$ is closed and has non-empty interior for all $\theta\leq\theta'$. Let $p$ denote the unique Lagrange multiplier associated with $(\vect{u},\rho)$ such that $(\vect{u},\rho,p)$ satisfy the first-order optimality conditions \cref{FOC1}--\cref{FOC3}. 
	 
Consider the finite element spaces $\vect{X}_h \subset \vect{Z}_h$,  $C_{\gamma,h} \subset C_\gamma$,  and $M_h \subset L^2_0(\Omega)$ and suppose that the assumptions \labelcref{ass:gh}--\labelcref{ass:dense} hold. 
	
Then, there exists an $\bar h >0$ such that, for $h\leq\bar h$, $h \to 0 $, there is a  family  of solutions $(\vect{u}_{h}, \rho_{h}, p_{h}) \in \vect{X}_{h,\vect{g}_h} \times C_{\gamma,{h}} \times M_{h}$  to the following discretized first-order optimality conditions
\begin{align}
a_h(\vect{u}_{h},\vect{v}_{h}; \rho_h) + b(\vect{v}_{h}, p_{h}) &= l_h(\vect{v}_{h}; \vect{g}_h) && \text{for all} \; \vect{v}_{h} \in \vect{X}_{h,0}, \tag{FOC1-h} \label{FOC1h}\\
b(\vect{u}_{h},q_{h}) &=0 && \text{for all} \; q_{h} \in M_{h}, \tag{FOC2-h} \label{FOC2h}\\
c(\rho_{h}, \eta_{h} - \rho_{h}; \vect{u}_h) &\geq 0 && \text{for all} \; \eta_{h} \in C_{\gamma,{h}}, \tag{FOC3-h} \label{FOC3h}
\end{align}
such that, $\|\vect{u}-\vect{u}_{h}\|_{H^1_{\vect{g}}(\mathcal{T}_h)} \to 0$, $\rho_{h} \to \rho$ strongly in $L^s(\Omega)$, $s \in [1,\infty)$,  and $p_{h} \to p$ strongly in $L^2(\Omega)$ as $h \to 0$.
\end{theorem}

\begin{remark}
Convergence of $\vect{u}_{h}$ in the norm $\|\cdot \|_{H^1_{\vect{g}}(\mathcal{T}_h)}$ ensures that the Dirichlet boundary condition is correctly satisfied in the limit. This cannot be  immediately deduced if $\vect{u}_{h}$ only strongly converges in the norm  $\|\cdot \|_{H^1(\mathcal{T}_h)}$.
\end{remark}

We first introduce some auxiliary propositions to facilitate the proof of \cref{th:FEMexistence}.   In all the following propositions and corollaries, we assume that the conditions in \cref{th:FEMexistence} hold and we fix an isolated minimizer $(\vect{u},\rho) \in \Hu \times C_\gamma$ of \cref{borrvallmin}.  We also fix an $r>0$ such that $(\vect{u},\rho)$ is the unique local minimizer of \cref{borrvallmin} in $B_{r, \Hdiv \times L^2(\Omega)}(\vect{u},\rho) \cap (\Hgdiv \times C_{\gamma})$, where
\begin{align}
\begin{split}
&B_{r, \Hdiv \times L^2(\Omega)}(\vect{u},\rho) \\
&\coloneqq \{\vect{v} \in \Hdiv, \; \eta \in C_\gamma: \|\vect{u} - \vect{v}\|_{\Hdiv} + \|\rho - \eta\|_{L^2(\Omega)}\leq r \}.
\end{split}
\end{align}
Such an $r$ is guaranteed to exist by the definition of an isolated minimizer and the extension of $J$ in \cref{Hdiv-J} to functions $\vect{v} \in \Hgdiv$ such that $\vect{v} \not\in H^1_{\vect{g}}(\Omega)^d$. We also define $B_{r,\Hdiv}(\vect{u})$ and $B_{r,L^2(\Omega)}(\rho)$ by
\begin{align}
B_{r, \Hdiv}(\vect{u}) &\coloneqq \{\vect{v} \in \Hdiv: \|\vect{u} - \vect{v}\|_{\Hdiv} \leq r \},\\
B_{r,L^2(\Omega)}(\rho)& \coloneqq \{ \eta \in C_\gamma:\|\rho - \eta\|_{L^2(\Omega)} \leq r \}.
\end{align}
We note that 
\begin{align*}
&(\Hgdiv \cap B_{r/2,\Hdiv}(\vect{u})) \times (C_{\gamma}\cap B_{r/2,L^2(\Omega)}(\rho)) \\
&\indent \subset B_{r, \Hdiv \times L^2(\Omega)}(\vect{u},\rho) \cap (\Hgdiv \times C_{\gamma} )
\end{align*}
and hence  $(\vect{u},\rho)$ is also the unique minimizer in $(\Hgdiv \cap B_{r/2,\Hdiv}(\vect{u})) \times (C_{\gamma}  \cap B_{r/2,L^2(\Omega)}(\rho))$. 
\begin{remark}
The extension of $J$ in \cref{Hdiv-J} to functions $\vect{v} \in \Hgdiv$ such that $\vect{v} \not\in H^1_{\vect{g}}(\Omega)^d$ means that functions $\vect{v} \not\in H^1_{\vect{g}}(\Omega)^d$ cannot be minimizers. 
\end{remark} 
\begin{remark}
We make the assumption that $\rho \in C_\gamma$ is isolated with respect to the $L^2$-norm (as opposed to the $L^\infty$-norm). This is a stronger isolation assumption as discussed in \cite[Rem.~7]{Papadopoulos2021b}. However, it is equivalent to being isolated with respect to any $L^s$-norm for $s \in [1,\infty)$ as discussed in \cite[Rem.~8]{Papadopoulos2021b}. As far as we are aware, the $L^2$-isolation assumption is valid for all practical problems found in the literature, in particular it holds for the example found in \cref{sec:numerical}.
\end{remark}
\begin{remark}
We are required to make the assumption that $\vect{u}$ is isolated with respect to the $\Hdiv$-norm (as opposed to the $H^1$-norm), i.e.~for all $\vect{v} \in \Hgdiv$ such that $\vect{v} \in B_{r/2, \Hdiv}(\vect{u})$, then $\vect{v}$ cannot be part of a minimizing pair. This is necessary to construct the discretized problem \cref{BPh} below. This is a stronger isolation assumption than isolation with respect to the $H^1$-norm. However, in examples found in the literature and, in particular, the example we consider in \cref{sec:numerical}, the velocity minimizers are always isolated with respect to the $\Hdiv$-norm. 
\end{remark}

\begin{proposition}[Weak convergence of ($\vect{u}_h, \rho_h)$ in $\Hdiv \times L^2(\Omega)$]
\label{prop:FEMconvergence}
For a given $h > 0$, consider the finite-dimensional optimization problem: find $(\vect{u}_h,\rho_h) \in (\vect{U}_{h,\vect{g}_h} \cap B_{r/2,\Hdiv}(\vect{u})) \times (C_{\gamma, h} \cap B_{r/2,L^2(\Omega)}(\rho))$ that minimizes
\begin{align}
J_h(\vect{u}_h, \rho_h). \tag{BP-h} \label{BPh}
\end{align}
Then, a global minimizer $(\vect{u}_h,\rho_h)$ of \cref{BPh} exists and there exist subsequences (up to relabeling) such that as $h \to 0$:
\begin{align}
\vect{u}_h  &\weak \vect{u} \; \text{weakly in} \; \Hdiv,\\
\vect{u}_h &\to \vect{u} \;\; \text{strongly in} \;\; L^q(\Omega)^d,\\
\vect{u}_h &\to \vect{u} \;\; \text{strongly in} \;\; L^r(\partial \Omega)^d,\\
\rho_h  &\weakstar \rho \; \text{weakly-* in} \; L^\infty(\Omega),\\
\rho_h  &\weak \rho \; \text{weakly in} \; L^s(\Omega),\; s \in [1,\infty),
\end{align}
where $1 \leq q, r < \infty$ in two dimensions and $1 \leq q < 6$, $1 \leq r <4$ in three dimensions.
\end{proposition}
\begin{proof}
The functional $J_h$ is continuous and 
\begin{align}
(\vect{U}_{h,\vect{g}_h} \cap B_{r/2,\Hdiv}(\vect{u})) \times (C_{\gamma, h} \cap B_{r/2,L^2(\Omega)}(\rho))
\end{align}
is a finite-dimensional, closed and bounded set and,  for sufficiently small $h$, non-empty,  therefore, sequentially compact by the Heine--Borel theorem \cite[Th.~11.18]{Fitzpatrick2009}. Hence, $J_h$ attains its infimum in $(\vect{U}_{\vect{g}_h,h} \cap B_{r/2,\Hdiv}(\vect{u})) \times (C_{\gamma, h} \cap B_{r/2,L^2(\Omega)}(\rho))$ and, therefore, a global minimizer $(\vect{u}_h, \rho_h)$ exists. 

By a corollary of Kakutani's Theorem \cite[Th.~A.65]{Fonseca2006}, if a Banach space is reflexive then every norm-closed, bounded and convex subset of the Banach space is weakly compact and thus, by the Eberlein--\v{S}mulian theorem \cite[Th.~A.62]{Fonseca2006}, sequentially weakly compact. It can be checked that $\Hdiv \cap B_{r/2,\Hdiv}(\vect{u})$ and $C_\gamma \cap B_{r/2,L^2(\Omega)}(\rho)$ are norm-closed, bounded and convex subsets of the reflexive Banach spaces $\Hdiv$ and $L^2(\Omega)$, respectively. Therefore,  $\Hdiv \cap B_{r/2,\Hdiv}(\vect{u})$ is weakly sequentially compact in $\Hdiv$  and $C_\gamma \cap B_{r/2,L^2(\Omega)}(\rho)$ is weakly sequentially compact in $L^2(\Omega)$. 

Hence we extract subsequences (up to relabeling), $(\vect{u}_h)$ and $(\rho_h)$ of the sequence generated by the global minimizers of \cref{BPh} such that
\begin{align}
\vect{u}_h &\weak \hat{\vect{u}} \in \Hdiv \cap B_{r/2,\Hdiv}(\vect{u}) \; \text{weakly in} \; \Hdiv,\\
\rho_h &\weak \hat{\rho} \in C_\gamma \cap B_{r/2,L^2(\Omega)}(\rho) \; \text{weakly in} \; L^2(\Omega).
\end{align}

By assumption \labelcref{ass:dense}, there exists a sequence of finite element functions $\tilde{\rho}_h \in C_{\gamma,h}$ that strongly converges to  $\rho$ in $L^2(\Omega)$. Moreover, \cref{lemma:strongv} implies the existence of a sequence $(\vect{\tilde{u}}_h) \in \vect{U}_{\vect{g}_h,h}$ that satisfies $\| \vect{u} - \vect{\tilde{u}}_h \|_{H^1_{\vect{g}}(\mathcal{T}_h)} \to 0$.

We now wish to identify the limits $\hat{\vect{u}}$ and $\hat{\rho}$. Consider the following bound:
\begin{align}
\label{femconvergence1}
\begin{split}
&2|J_h(\vect{\tilde{u}}_h, \tilde{\rho}_h) - J(\vect{u},\rho)|\\
&\indent \leq \int_\Omega |(\alpha(\rho)-\alpha(\tilde{\rho}_h)) |\vect{u}|^2 | + | \alpha(\tilde{\rho}_h) (|\vect{u}|^2 - |\vect{\tilde{u}}_h|^2)|+ 2 |\vect{f} \cdot (\vect{u} - \vect{\tilde{u}}_h)| \; \dx\\
&\indent \indent +\nu \sum_{K\in\mathcal{T}_h} \int_K \left||\nabla \vect{u} |^2-|\nabla \vect{\tilde{u}}_h|^2 |  \right|  \; \dx\\
& \indent \indent +  \nu  \sum_{F \in \mathcal{F}^i_h} \int_F \sigma h_F^{-1}| \lsb \vect{\tilde{u}}_h \rsb_F|^2 \, \mathrm{d}s +  \nu  \sum_{F \in \mathcal{F}^\partial_h}  \int_F \sigma h_F^{-1}| \lsb \vect{\tilde{u}}_h - \vect{g}_h \rsb_F|^2 \, \mathrm{d}s \\
& \indent \indent  +2  \nu  \sum_{F \in \mathcal{F}^i_h}\int_F |\lcb \nabla \vect{\tilde{u}}_h \rcb_F : \lsb \vect{\tilde{u}}_h \rsb_F |\, \mathrm{d}s\\
& \indent \indent + 2  \nu  \sum_{F \in \mathcal{F}^\partial_h}  \int_F |\lcb \nabla \vect{\tilde{u}}_h \rcb_F : \lsb \vect{\tilde{u}}_h - \vect{g}_h \rsb_F | \,\mathrm{d}s. 
\end{split}
\end{align}
For all $\vect{v} \in H^1(\mathcal{T}_h)^d$, $\tens{\Phi} \in (X_{\DG_k})^{d \times d}$, $h > 0$, we have that \cite[Lem.~7]{Buffa2009}
\begin{align}
\label{ineq:buffa}
\begin{split}
& \sum_{F \in \mathcal{F}_h} \int_F  |\lcb \tens{\Phi} \rcb_F : \lsb \vect{v} \rsb_F|\,\mathrm{d}s\\
&\leq C\left( \sum_{F \in \mathcal{F}_h} \int_F h_F^{-1} | \lsb \vect{v} \rsb_F|^2 \, \mathrm{d}s\right)^{1/2}\left(\sum_{K\in\mathcal{T}_h}  \| \tens{\Phi}\|^2_{L^2(K)}\right)^{1/2},
\end{split}
\end{align}
for a constant $C$ that only depends on the mesh quality. Hence, we see that
\begin{align}
\label{femconvergence1.5}
\begin{split}
&2|J_h(\vect{\tilde{u}}_h, \tilde{\rho}_h) - J(\vect{u},\rho)|\\
&\indent \leq L_\alpha  \|\vect{u}\|^2_{L^4(\Omega)}\| \tilde{\rho}_h - \rho\|_{L^2(\Omega)} + 2 \|\vect{f}\|_{L^2(\Omega)}\|\vect{\tilde{u}}_h - \vect{u}\|_{L^2(\Omega)} \\
&\indent \indent+ \bar{\alpha} \|\vect{\tilde{u}}_h -\vect{u}\|_{L^2(\Omega)}( \|\vect{\tilde{u}}_h - \vect{u}\|_{L^2(\Omega)} + 2\|\vect{u}\|_{L^2(\Omega)})\\
& \indent \indent +\nu \sum_{K \in \mathcal{T}_h}  \|\nabla \vect{\tilde{u}}_h - \nabla \vect{u} \|_{L^2(K)}( \|\nabla \vect{\tilde{u}}_h - \nabla \vect{u}\|_{L^2(K)} + 2\|\nabla \vect{u}\|_{L^2(K)}) \\
& \indent \indent +  \nu  \sum_{F \in \mathcal{F}^i_h} \int_F \sigma h_F^{-1}| \lsb \vect{\tilde{u}}_h \rsb_F|^2 \mathrm{d}s +  \nu  \sum_{F \in \mathcal{F}^\partial_h}  \int_F \sigma h_F^{-1}| \lsb \vect{\tilde{u}}_h - \vect{g}_h \rsb_F|^2 \mathrm{d}s\\
& \indent \indent +  C  \nu  \left( \sum_{F \in \mathcal{F}^i_h} \int_F h_F^{-1}| \lsb \vect{\tilde{u}}_h \rsb_F|^2 \mathrm{d}s\right)^{1/2} \left(\sum_{K\in\mathcal{T}_h}  \| \nabla \vect{\tilde{u}}_h \|^2_{L^2(K)}\right)^{1/2}\\
& \indent \indent + C  \nu  \left( \sum_{F \in \mathcal{F}^\partial_h} \int_F h_F^{-1}| \lsb \vect{\tilde{u}}_h -\vect{g}_h\rsb_F|^2 \mathrm{d}s\right)^{1/2} \left(\sum_{K\in\mathcal{T}_h}  \| \nabla \vect{\tilde{u}}_h \|^2_{L^2(K)}\right)^{1/2},
\end{split}
\end{align}
where $L_\alpha$ denotes the Lipschitz constant for $\alpha$. Thanks to the strong convergence of $\vect{\tilde{u}}_h$ in the broken $H^1_{\vect{g}}$-norm to $\vect{u}$ and by assumption \labelcref{ass:gh}, from \cref{femconvergence1.5} we deduce that 
\begin{align*}
J_h(\vect{\tilde{u}}_h, \tilde{\rho}_h) \to J(\vect{u},\rho) \;\; \text{as} \;\; h \to 0. 
\end{align*}
Furthermore, for sufficiently small $h >0$ we note that 
\begin{align*}
(\vect{\tilde{u}}_h, \tilde{\rho}_h) \in (\vect{U}_{\vect{g}_h,h} \cap B_{r/2,\Hdiv}(\vect{u})) \times (C_{\gamma,h} \cap B_{r/2,L^2(\Omega)}(\rho)).
\end{align*}
Therefore, since $(\vect{u}_h,\rho_h)$ is a global minimizer of \cref{BPh} in $B_{r/2,\Hdiv}(\vect{u})) \times (C_{\gamma,h} \cap B_{r/2,L^2(\Omega)}(\rho))$,
\begin{align}
J_h(\vect{u}_h,\rho_h) \leq J_h(\vect{\tilde{u}}_h, \tilde{\rho}_h).
\end{align}
By taking the limit as $h \to 0$ and utilizing the strong convergence of $\vect{\tilde{u}}_h$ and $\tilde{\rho}_h$ to $\vect{u}$ and $\rho$, respectively, we see that
\begin{align}
\lim_{h \to 0} J_h(\vect{u}_h, \rho_h) \leq J(\vect{u}, \rho). \label{eq:femapprox1}
\end{align}

By assumption \labelcref{ass:dense}, for every $q \in L^2_0(\Omega)$, there exists a sequence of $\tilde{q}_h \in M_h$ such that $\tilde{q}_h \to q$ strongly in $L^2(\Omega)$. Since $\vect{u}_h \weak \hat{\vect{u}}$ weakly in $\Hdiv$ and $\vect{u}_h \in \vect{U}_{\vect{g}_h,h}$, we see that
\begin{align}
b(\hat{\vect{u}}, q) = \lim_{h \to 0} b(\vect{u}_h, \tilde{q}_h) +  \lim_{h \to 0} b(\vect{u}_h, q - \tilde{q}_h)= 0 \;\; \text{for all} \;\; q \in L^2_0(\Omega). \label{hatudivfree}
\end{align}
Hence, $\hat{\vect{u}}$ is pointwise divergence-free. The final step to identify $\hat{\vect{u}}$ as $\vect{u}$ is to show that $\hat{\vect{u}} \in H^1_{\vect{g}}(\Omega)^d$. Now, the sequence $(\vect{u}_h)$ also defines a bounded sequence in $H^1(\mathcal{T}_h)^d$ such that
\begin{align*}
\sup_{h > 0} \left[ \|\vect{u}_h \|_{L^1(\Omega)} + |\vect{u}_h|_{H^1(\mathcal{T}_h)} \right] < + \infty.
\end{align*}
Hence, by the compact embedding lemma, as found in Buffa and Ortner \cite[Lem.~8]{Buffa2009}, there exists a subsequence (up to relabeling) and a limit $\hat{\vect{w}} \in H^1(\Omega)^d$ such that 
\begin{align}
\vect{u}_h \to \hat{\vect{w}}  \;\; \text{strongly in} \;\; L^q(\Omega)^d, \label{hatuh1}
\end{align}
where $1 \leq q < \infty$ in two dimensions and $1 \leq q < 6$ in three dimensions. By the uniqueness of limits $\hat{\vect{w}} = \hat{\vect{u}}$ a.e.~in $\Omega$ and thus $\hat{\vect{u}} \in H^1(\Omega)^d$. Moreover, the same compact embedding lemma implies that 
\begin{align}
\vect{u}_h \to \hat{\vect{u}} \;\; \text{strongly in} \;\; L^r(\partial \Omega)^d,
\end{align}
where $1 \leq r < \infty$ in two dimensions and $1 \leq r< 4$ in three dimensions. If $\|\vect{u}_h - \vect{g}\|_{L^2(\partial \Omega)} \not \to 0$, then $J_h(\vect{u}_h, \rho_h) \to +\infty$. Since $(\vect{u}_h)$ is a bounded sequence, we must have that $\|\vect{u}_h - \vect{g}\|_{L^2(\partial \Omega)} \to 0$. Hence,
\begin{align}
\| \hat{\vect{u}} - \vect{g} \|_{L^2(\partial \Omega)} &\leq \| \hat{\vect{u}} - \vect{u}_h \|_{L^2(\partial \Omega)} + \|\vect{u}_h - \vect{g} \|_{L^2(\partial \Omega)} \to 0.
 \label{hatug}
\end{align}

Thus,  \cref{hatudivfree}, \cref{hatuh1}, and \cref{hatug} imply that $\hat{\vect{u}} \in  \Hu \cap B_{r/2,\Hdiv}(\vect{u})$.  

In order to identify the weak limit $(\vect{\hat{u}}, \hat{\rho})$ with the isolated minimizer $(\vect{u},\rho)$, we require a weak lower semicontinuity result. Consider the following decomposition of the functional $J_h(\vect{u}_h,\rho_h) = J_{1,h}(\vect{u}_h,\rho_h) + J_{2,h}(\vect{u}_h)$ where
\begin{align}
J_{1,h}(\vect{u}_h,\rho_h) = \frac{1}{2}\int_\Omega \alpha(\rho_h)|\vect{u}_h|^2 - 2\vect{f}\cdot \vect{u}_h \, \dx,
\end{align}
and $J_{2,h}(\vect{u}_h) = J_h(\vect{u}_h,\rho_h) - J_{1,h}(\vect{u}_h,\rho_h)$. It follows from \cref{hatuh1} and a small modification to the proof in \cite[Th.~3.1]{Borrvall2003} that
\begin{align}
J_{1,h}(\vect{\hat{u}},\hat{\rho}) \leq \liminf_{h\to 0} J_{1,h}(\vect{u}_h,\rho_h).
\end{align}
Moreover, it follows from a convergence result in Buffa and Ortner \cite[Th.~6.1]{Buffa2009} that 
\begin{align}
J_{2,h}(\vect{u}_h) \to \frac{\nu}{2}\int_\Omega |\nabla \vect{\hat{u}}|^2 \, \dx. 
\end{align}
Hence, we have the following weak lower semicontinuity result:
\begin{align}
J(\vect{\hat{u}}, \hat{\rho}) \leq \liminf_{h \to 0} J_h(\vect{u}_h, \rho_h). \label{lowersemi2}
\end{align}
Since $(\vect{u},\rho)$ is the unique minimizer  of \cref{borrvallmin} in 
\begin{align*}
(\Hgdiv \cap B_{r/2,\Hdiv}(\vect{u})) \times (C_\gamma \cap B_{r/2,L^2(\Omega)}(\rho)),
\end{align*}
we see that $J(\vect{u}, \rho) \leq J(\vect{\hat{u}},\hat{\rho})$. Hence, from \cref{eq:femapprox1} and \cref{lowersemi2}, it follows that
\begin{align}
J(\vect{\hat{u}}, \hat{\rho}) = J(\vect{u}, \rho).
\end{align} 
Since $(\vect{u}, \rho)$ is the unique minimizer in the spaces we consider, we identify the limits $\hat{\vect{u}}$ and $\hat{\rho}$ as $\vect{u}$ and $\rho$, respectively, and state that $\vect{u}_h \weak \vect{u}$ weakly in $\Hdiv$,  $\vect{u}_h \to \vect{u}$ strongly in $L^q(\Omega)^d$, $\vect{u}_h \to \vect{u}$ strongly in $L^r(\partial \Omega)^d$ and $\rho_h \weak \rho$ weakly in $L^2(\Omega)$, where $1\leq q,r < \infty$ in two dimensions and $1 \leq q <6$, $1 \leq r <4$ in three dimensions.

We note that by the Banach--Alaoglu theorem \cite[Th.~A.52]{Fonseca2006},  the closed unit ball of the dual space of a normed vector space, (for example $L^1(\Omega)$), is compact in the weak-* topology. Hence we also find a subsequence such that $\rho_h \weakstar \hat{\rho} \in C_\gamma \cap \{\eta: \|\rho-\eta\|_{L^\infty(\Omega)} \leq r/2\}$ weakly-* in $L^\infty(\Omega)$. By the uniqueness of the weak limit, we identify $\hat{\rho} = \rho$ a.e.~in $\Omega$ and, thus, we deduce that $\rho_h \weakstar \rho$ weakly-* in $L^\infty(\Omega)$. Consequently, $\rho_h \weak \rho$ weakly in $L^s(\Omega)$ for all $s \in [1,\infty)$. 
\end{proof}

\begin{proposition}[Strong convergence of $\rho_h$ in $L^s(\Omega)$, $s\in[1,\infty)$]
\label{prop:FEMconvergence1.5}
There exists a subsequence of minimizers, $(\rho_h)$, of \cref{BPh} such that
\begin{align}
\rho_h \to \rho \; \text{strongly in} \; L^s(\Omega), \;\; s \in [1,\infty).
\end{align}
\end{proposition}
The proof of \cref{prop:FEMconvergence1.5}, with some small modifications, can be found in \cite[Prop.~5]{Papadopoulos2021b}. For the convenience of the reader, we reproduce the proof in \cref{app:prop1.5} with the necessary changes.

\begin{proposition}[Strong convergence of $\vect{u}_h$ in the ${H^1_{\vect{g}}(\mathcal{T}_h)}$-norm]
\label{prop:FEMconvergence2}
There exists a subsequence of minimizers, $(\vect{u}_h)$, of \cref{BPh} such that
\begin{align}
\|\vect{u} - \vect{u}_h\|_{H^1_{\vect{g}}(\mathcal{T}_h)} \to 0.
\end{align}
\end{proposition}
\begin{proof}
We note that $\vect{U}_{h,\vect{g}_h} \cap B_{r/2,\Hdiv}(\vect{u})$ is a convex set, and hence for any $\vect{w}_h \in \vect{U}_{h,\vect{g}_h} \cap B_{r/2,\Hdiv}(\vect{u})$, $t \in [0,1]$, we have that $ \vect{u}_h  + t(\vect{w}_h - \vect{u}_h) \in \vect{U}_{h,\vect{g}_h} \cap B_{r/2,\Hdiv}(\vect{u})$. Since $(\vect{u}_h, \rho_h)$ is a global minimizer of \cref{BPh}, we note that
\begin{align}
\frac{1}{t} \left[ J_h(\vect{u}_h+ t(\vect{w}_h - \vect{u}_h),  \rho_h) - J_h(\vect{u}_h, \rho_h) \right] \geq 0.
\end{align}
By taking the limit $t \to 0$, a calculation shows that, for all $\vect{w}_h \in \vect{U}_{h,\vect{g}_h} \cap B_{r/2,\Hdiv}(\vect{u})$,
\begin{align}
a_h(\vect{u}_h, \vect{w}_h - \vect{u}_h; \rho_h) \geq l_h(\vect{w}_h - \vect{u}_h; \vect{g}_h).  \label{velerror1}
\end{align}
We note that $\vect{w}_h - \vect{u}_h \in \vect{U}_{h,0}$. Hence, from \cref{prop:consistencyah} and \cref{lem:pointwisedivfree}, we deduce that
\begin{align}
a_h(\vect{u}, \vect{w}_h - \vect{u}_h; \rho) = l_h(\vect{w}_h - \vect{u}_h; \vect{g}). \label{velerror2}
\end{align}
Therefore, from \cref{velerror1} and \cref{velerror2}, we see that
\begin{align}
\label{velerror3}
\begin{split}
a_h(\vect{u}_h, \vect{u}_h - \vect{w}_h; \rho_h) &\leq a_h(\vect{u}, \vect{u}_h - \vect{w}_h; \rho) \\
& \indent + l_h(\vect{u}_h-\vect{w}_h; \vect{g}_h) -  l_h(\vect{u}_h-\vect{w}_h; \vect{g}).
\end{split}
\end{align}
Hence, by subtracting $a_h(\vect{w}_h, \vect{u}_h - \vect{w}_h; \rho_h)$ from both sides of \cref{velerror3}, and utilizing the coercivity of $a_h(\cdot, \cdot; \cdot)$ as stated in \cref{prop:coercivity}, we have that
\begin{align}
\label{velerror4}
\begin{split}
c_a \| \vect{u}_h - \vect{w}_h \|^2_{H^1(\mathcal{T}_h)} &\leq a_h(\vect{u}, \vect{u}_h - \vect{w}_h; \rho) - a_h(\vect{w}_h, \vect{u}_h - \vect{w}_h; \rho_h) \\
&\indent + l_h(\vect{u}_h-\vect{w}_h; \vect{g}_h) -  l_h(\vect{u}_h-\vect{w}_h; \vect{g}).
\end{split}
\end{align}
Now by assumption \labelcref{ass:M3}, for all $F \in \mathcal{F}^\partial_h$ there exists a $c>0$ such that $h^{-1}_F \leq ch^{-1}$, where $c$ depends on the mesh regularity. By taking the absolute value of the right-hand side of \cref{velerror4}, collecting terms, utilizing the inequality \cref{ineq:buffa}, and the boundedness of $a_h$ by $C_a$ as stated in \cref{prop:coercivity}, we have that
\begin{align}
\begin{split}
c_a\|\vect{u}_h - \vect{w}_h\|^2_{H^1(\mathcal{T}_h)}&\leq  \bar \alpha \|\vect{u}-\vect{w}_h\|_{L^2(\Omega)}\|\vect{u}_h-\vect{w}_h\|_{L^2(\Omega)}\\
&\indent+  \|(\alpha(\rho) - \alpha(\rho_h)) \vect{u} \|_{L^{2}(\Omega)} \|\vect{u}_h-\vect{w}_h\|_{L^2(\Omega)}\\
&\indent+ C_a \|\vect{u} - \vect{w}_h\|_{H^1(\mathcal{T}_h)}\|\vect{u}_h-\vect{w}_h\|_{H^1(\mathcal{T}_h)}\\
& \indent + C h^{-1} \| \vect{g} - \vect{g}_h \|_{L^2(\partial \Omega)} \| \vect{u}_h - \vect{w}_h \|_{L^2( \partial \Omega)}\\
&  \indent + C h^{-1} \| \vect{g} - \vect{g}_h \|_{L^2(\partial \Omega)} \| \vect{u}_h - \vect{w}_h \|_{H^1(\mathcal{T}_h)},
\end{split}
\label{velerror7}
\end{align}
for some constant $C = C(\sigma)$ that also depends on the mesh regularity. 

We note that  $\|\vect{u}-\vect{w}_h\|_{L^2(\Omega)} \leq \|\vect{u}_h - \vect{w}_h\|_{H^1(\mathcal{T}_h)}$ by definition. Moreover, by the broken trace theorem as found in Buffa and Ortner \cite[Th.~4.4]{Buffa2009}, there exists a constant $C_{\mathrm{BT}}$ such that, for all $\vect{v} \in H^1(\mathcal{T}_h)^d$, $d \in \{2,3\}$ we have
\begin{align}
\| \vect{v} \|_{L^2(\partial \Omega)} \leq C_{\mathrm{BT}} \| \vect{v} \|_{H^1(\mathcal{T}_h)}. 
\end{align}
Therefore, by bounding the $L^2(\Omega)$ and $L^2(\partial \Omega)$-norms of $\vect{u}_h - \vect{w}_h$ above by the broken $H^1$-norm, and dividing through by $c_a\|\vect{u}_h - \vect{w}_h\|_{H^1(\mathcal{T}_h)}$ we see that
\begin{align}
\begin{split}
\|\vect{u}_h - \vect{w}_h\|_{H^1(\mathcal{T}_h)}&\leq  C \|\vect{u}-\vect{w}_h\|_{L^2(\Omega)}
+  C\|(\alpha(\rho) - \alpha(\rho_h)) \vect{u} \|_{L^{2}(\Omega)} \\
&\indent+ C \|\vect{u} - \vect{w}_h\|_{H^1(\mathcal{T}_h)}+ C  h^{-1} \| \vect{g} - \vect{g}_h \|_{L^2(\partial \Omega)},\\
\end{split} \label{velerror6}
\end{align}
for some constant $C$ that depends on $c_a, C_a, C_{\mathrm{BT}}, \bar \alpha, \sigma$ and the mesh regularity. 

For sufficiently small $h$, we note that $\vect{\tilde{u}}_h \in \vect{U}_{h,\vect{g}_h} \cap B_{r/2,\Hdiv}(\vect{u})$ (where $\vect{\tilde{u}}_h$ is defined in \cref{lemma:strongv}) and $\vect{\tilde{u}}_h \to \vect{u}$ strongly in $H^1(\mathcal{T}_h)^d$. Fix $\vect{w}_h =\vect{\tilde{u}}_h$.

From \cref{prop:FEMconvergence1.5}, we know that there exists a subsequence (not indicated) such that  $\rho_h \to \rho$ strongly in $L^4(\Omega)$. We now observe that
\begin{align}
\label{velerror5}
\|(\alpha(\rho) - \alpha(\rho_h)) \vect{u} \|_{L^{2}(\Omega)} 
\leq L_\alpha  \|\rho - \rho_h\|_{L^4(\Omega)} \| \vect{u}\|_{L^{4}(\Omega)},
\end{align}
where $L_\alpha$ is the Lipschitz constant for $\alpha$. $\| \vect{u}\|_{L^{4}(\Omega)}$ is bounded for $d \in \{2,3\}$ thanks to the Sobolev embedding theorem. Hence, by taking the limit as $h\to 0$ in \cref{velerror6}, from \labelcref{ass:gh}, \cref{lemma:strongv}, and \cref{velerror5}, we deduce that $\| \vect{u} - \vect{u}_h\|_{H^1(\mathcal{T}_h)} \to 0$ as $h \to 0$. In \cref{prop:FEMconvergence}, we showed that $\vect{u}_h \to \vect{u}$ strongly in $L^2(\partial \Omega)^d$. Hence, we conclude that $\| \vect{u} - \vect{u}_h\|_{H^1_{\vect{g}}(\mathcal{T}_h)} \to 0$ as $h \to 0$. 
\end{proof}

\begin{proposition}[Discretized first-order optimality conditions]
\label{prop:FEMconvergence3}
There exists an $\bar h > 0$ such that for all $h < \bar h$, there exists a unique Lagrange multiplier $p_h \in M_h$ such that the functions $(\vect{u}_h, \rho_h)$, that minimize \cref{BPh}, satisfy the first-order optimality conditions \cref{FOC1h}--\cref{FOC3h}.
\end{proposition}
\begin{proof}
From \cref{prop:FEMconvergence2}, we know that $\vect{u}_h \to \vect{u}$ strongly in $H^1(\mathcal{T}_h)^d$. Hence by definition of strong convergence, there exists an $\bar h_1>0$ such that, for all $h \leq \bar h_1$, $\|\vect{u} - \vect{u}_{h}\|_{H^1(\mathcal{T}_h)} \leq r/4$. Moreover, since $\vect{u} \in \Hu$, we have that $\divv{\vect{u}} = 0$ a.e.~in $\Omega$ and by \cref{lem:pointwisedivfree}, we have that $\divv{\vect{u}_h} = 0$ a.e.~in $\Omega$. Therefore,
\begin{align}
\begin{split}
\| \vect{u} - \vect{u}_h \|^2_{\Hdiv} &= \| \vect{u} - \vect{u}_h \|^2_{L^2(\Omega)} + \| \divv{\vect{u} - \vect{u}_h} \|^2_{L^2(\Omega)}\\
& = \| \vect{u} - \vect{u}_h \|^2_{L^2(\Omega)} \leq \|\vect{u} - \vect{u}_{h}\|^2_{H^1(\mathcal{T}_h)} \leq r^2/16.
\end{split}
\end{align}

Hence, for each $\vect{v}_h \in \vect{U}_{h,0}$, if $|t|<r/(4\|\vect{v}_h\|_{H^1(\mathcal{T}_h)})$ then $\vect{u}_h + t \vect{v}_h \in \vect{U}_{h,\vect{g}_h} \cap B_{r/2,\Hdiv}(\vect{u})$. From \cref{prop:FEMconvergence1.5} we have that $\rho_h \to \rho$ strongly in $L^2(\Omega)$. Hence, there exists an $\bar h_2>0$ such that, for all $h \leq \bar h_2$, $\|\rho - \rho_{h}\|_{L^2(\Omega)} \leq r/4$. Therefore, for each $\eta_h \in C_{\gamma,h}$, if $0 < t <r/(4\|\eta_h-\rho_h\|_{L^2(\Omega)})$ then $\rho_h + t (\eta_h- \rho_h) \in C_{\gamma,h} \cap B_{r/2,L^2(\Omega)}(\rho)$. Let $\bar h = \min \{\bar h_1, \bar h_2\}$ and consider $h \leq \bar h$. 

Since $(\vect{u}_h, \rho_h)$ is a global minimizer of \cref{BPh}, then, for all $\vect{v}_h \in \vect{U}_{h,0}$, if $|t|<r/(4\|\vect{v}_h\|_{H^1(\mathcal{T}_h)})$ we have
\begin{align}
\frac{1}{t} \left[ J_h(\vect{u}_h+ t\vect{v}_h,  \rho_h) - J_h(\vect{u}_h, \rho_h) \right] \geq 0.
\end{align}
By considering the limits for $t \to 0_+$ and $t \to 0_-$, we have that, for all $\vect{v}_h \in \vect{U}_{h,0}$,
\begin{align}
a_h(\vect{u}_h, \vect{v}_h; \rho_h) = l_h(\vect{v}_h; \vect{g}_h). 
\label{focproof1}
\end{align}
From \cref{focproof1}, the existence of a unique $p_h \in M_h$ such that $(\vect{u}_h, \rho_h, p_h)$ satisfy \cref{FOC1h}--\cref{FOC2h} follows from the inf-sup condition \labelcref{ass:discreteinfsup} and the argument can be found in \cite[Prop.~2]{Papadopoulos2021b}.

Similarly, since $(\vect{u}_h, \rho_h)$ is a global minimizer of \cref{BPh}, then, for all $\eta_h \in C_{\gamma,h}$, if $0 < t <r/(4\|\eta_h-\rho_h\|_{L^2(\Omega)})$ we have
\begin{align}
\frac{1}{t} \left[ J_h(\vect{u}_h,  \rho_h+t(\eta_h - \rho_h)) - J_h(\vect{u}_h, \rho_h) \right] \geq 0.
\end{align}
By taking the limit as $t \to 0$, we deduce that \cref{FOC3h} holds.
\end{proof}

\begin{proposition}[Strong convergence of $p_h$ in $L^2(\Omega)$]
\label{prop:FEMconvergence4}
There is a subsequence of the unique $p_h \in M_h$ defined in \cref{prop:FEMconvergence3} that converges strongly in $L^2(\Omega)$ to the $p\in L^2_0(\Omega)$ that solves \cref{FOC1}--\cref{FOC3} for the given isolated minimizer $(\vect{u}, \rho)$.
\end{proposition}
\begin{proof}
The inf-sup condition \labelcref{ass:discreteinfsup} for $M_h$ and $\vect{X}_{h,0}$ implies that, for any $q_h \in M_h$,
\begin{align*}
c_b \|q_h-p_h\|_{L^2(\Omega)}  &\leq \sup_{\vect{w}_h \in \vect{X}_{h,0} \backslash \{0\}} \frac{b(\vect{w}_h,q_h-p_h)}{\|\vect{w}_h\|_{H^1(\mathcal{T}_h)}}\\
&=  \sup_{\vect{w}_h \in \vect{X}_{h,0}  \backslash \{0\}} \frac{b(\vect{w}_h,p-p_h)+b(\vect{w}_h,q_h-p)}{\|\vect{w}_h\|_{H^1(\mathcal{T}_h)}}\\
&\leq  \sup_{\vect{w}_h \in \vect{X}_{h,0}  \backslash \{0\}} \frac{|b(\vect{w}_h,p-p_h)|+|b(\vect{w}_h,q_h-p)|}{\|\vect{w}_h\|_{H^1(\mathcal{T}_h)}}.
\end{align*}
From \cref{prop:consistencyah} and \cref{prop:FEMconvergence3}, it follows that
\begin{align}
b(\vect{w}_h,p-p_h) = a_h(\vect{u}_h,\vect{w}_h; \rho_h) - a_h(\vect{u},\vect{w}_h; \rho)+l_h(\vect{w}_h; \vect{g}) - l_h(\vect{w}_h;\vect{g}_h).
\end{align}
Therefore, 
\begin{align}
\begin{split}
&c_b \|q_h-p_h\|_{L^2(\Omega)}  \\
&\leq \sup_{\vect{w}_h \in \vect{X}_{h,0} \backslash \{0\}} \frac{|a_h(\vect{u}_h,\vect{w}_h; \rho_h) - a_h(\vect{u},\vect{w}_h; \rho)|}{\|\vect{w}_h\|_{H^1(\mathcal{T}_h)}}\\
& \indent + \sup_{\vect{w}_h \in \vect{X}_{h,0} \backslash \{0\}}\frac{|l_h(\vect{w}_h; \vect{g}) - l_h(\vect{w}_h;\vect{g}_h)|+|b(\vect{w}_h,q_h-p)|}{\|\vect{w}_h\|_{H^1(\mathcal{T}_h)}}.
\end{split}
\label{pressureerror1}
\end{align}
By using the same argument we used to bound \cref{velerror7} by \cref{velerror6}, we see that \cref{pressureerror1} implies that
\begin{align}
\begin{split}
c_b \|q_h-p_h\|_{L^2(\Omega)} 
&\leq C \|(\alpha(\rho) - \alpha(\rho_h))\vect{u}\|_{L^2(\Omega)} + C \|\vect{u}-\vect{u}_h\|_{H^1(\mathcal{T}_h)}\\
&\indent + C h^{-1} \|\vect{g} - \vect{g}_h \|_{L^2(\partial \Omega)} + C_b \|p-q_h\|_{L^2(\Omega)}.
\end{split}
\end{align}
where $C_b$ is the boundedness constant for $b(\cdot, \cdot)$ and $C$ is dependent on $C_a$, $C_{\mathrm{BT}}$, $\bar \alpha$, $L_\alpha$, $\nu$, $\sigma$ and the mesh regularity. Hence, by an application on the Cauchy--Schwarz inequality,
\begin{align}
\begin{split}
\|p-p_h\|_{L^2(\Omega)}& \leq C \|(\alpha(\rho) - \alpha(\rho_h)) \vect{u} \|_{L^{2}(\Omega)} +C \|\vect{u}-\vect{u}_h\|_{H^1(\mathcal{T}_h)}\\
&\indent +Ch^{-1} \|\vect{g} - \vect{g}_h \|_{L^2(\partial \Omega)}+ C\|p-q_h\|_{L^2(\Omega)},
\end{split}
\end{align}
where $C$ is dependent on $c_b$, $C_b$, $C_a$, $C_{\mathrm{BT}}$, $\bar \alpha$, $L_\alpha$, $\nu$, $\sigma$ and the mesh regularity.

By assumption \ref{ass:dense}, there exists a sequence of finite element functions, $\tilde{p}_h \in M_h$ that satisfies $\tilde{p}_h \to p$ strongly in $L^2(\Omega)$. Let  $q_h = \tilde{p}_h$. We have already shown that $\vect{u}_h \to \vect{u}$ strongly in $H^1(\mathcal{T}_h)^d$ in \cref{prop:FEMconvergence2}. Similarly, in the proof of \cref{prop:FEMconvergence2} we also showed that $\|(\alpha(\rho) - \alpha(\rho_h)) \vect{u} \|_{L^{2}(\Omega)} \to 0$. By assumption \ref{ass:gh}, $h^{-1} \|\vect{g} - \vect{g}_h \|_{L^2(\partial \Omega)} \to 0$. Hence, we conclude that $p_h \to p$ strongly in $L^2(\Omega)$.
\end{proof}

We now have the required results to prove \cref{th:FEMexistence}.
\begin{proof}[Proof of \cref{th:FEMexistence}]
Fix an isolated minimizer $(\vect{u}, \rho)$ of \cref{borrvallmin} and its unique associated Lagrange multiplier $p$. By the results of Propositions \ref{prop:FEMconvergence}, \ref{prop:FEMconvergence1.5}, \ref{prop:FEMconvergence2}, and \ref{prop:FEMconvergence3}, there exists a mesh size $\bar h$ such that for, $h < \bar h$, there exists a sequence of finite element solutions $(\vect{u}_h, \rho_h, p_h) \in \vect{X}_{h,\vect{g}_h} \times C_{\gamma,h} \times M_h$ satisfying \cref{FOC1h}--\cref{FOC3h} that converges to $(\vect{u}, \rho, p)$. By taking a subsequence if necessary (not indicated),  \cref{prop:FEMconvergence1.5}, implies that $\rho_h \to \rho$ strongly in $L^s(\Omega)$, $s \in [1,\infty)$, \cref{prop:FEMconvergence2} implies that $\| \vect{u}- \vect{u}_h\|_{H^1_{\vect{g}}(\mathcal{T}_h)} \to 0$, and  \cref{prop:FEMconvergence4} implies that $p_h \to p$ strongly in $L^2(\Omega)$. 
\end{proof}

\section{Numerical results}
\label{sec:numerical}
In this section we consider an example of a Borrvall--Petersson topology optimization problem that supports two minimizers. Our goal is to discretize the velocity-pressure pair with a Brezzi--Douglas--Marini $\mathrm{BDM}_1 \times \mathrm{DG}_0$ discretization,  and the material distribution with a piecewise constant $\mathrm{DG}_0$ discretization,  and numerically verify the existence of two sequences of solutions to \cref{FOC1h}--\cref{FOC3h} that converge to the two different minimizers. Moreover, we compare the violation of the incompressibility constraint, measured by $\|\divv{\vect{u}_h}\|_{L^2(\Omega)}$, with the equivalent solution as computed via a Taylor--Hood $(\mathrm{CG}_2)^2 \times \mathrm{CG}_1$ discretization for the velocity-pressure pair on the same meshes.

We note that $\mathrm{DG}_0 \not\subset C_\gamma$ and  $\mathrm{DG}_0 \not\subset L^2_0(\Omega)$, as, in general, $\eta \in \mathrm{DG}_0$ does not satisfy $0 \leq \eta \leq 1$ a.e.~and $\int_\Omega \eta \, \dx \leq \gamma |\Omega|$, and, in general, $q \in \mathrm{DG}_0$ does not satisfy $\int_\Omega q \, \dx = 0$. However, the choice of optimization algorithm (described below) only finds solutions $\rho_h \in  C_\gamma$ and $p_h \in  L^2_0(\Omega)$.  Hence, we are effectively working with the correct conforming finite element spaces for the pressure and material distribution and, hence, our discretization restricted by the optimization strategy satisfies the conditions of \cref{th:FEMexistence}.

The problem is implemented with the finite element software Firedrake \cite{firedrake} and the computational domains are triangulated with simplices. The solutions are computed using the deflated barrier method \cite{Papadopoulos2021a}. The deflated barrier method reformulates \cref{FOC1h}--\cref{FOC3h} into a mixed complementarity problem and solves the nonlinear system with a primal-dual active set solver  that enforces the box constraints on the material distribution \cite{Benson2003}. The volume constraint is enforced via a one-dimensional Lagrange multiplier. The zero mean value constraint on the pressure is either enforced via a one-dimensional Lagrange multiplier or by orthogonalizing against the nullspace of constants.  The global nonlinear convergence is aided by the continuation of barrier terms. A key feature of the deflated barrier method is that it can systematically discover multiple solutions of topology optimization problems by utilizing the deflation technique \cite{Farrell2015, Farrell2019b}. In the BDM discretization, the linear systems arising in the deflated barrier method are solved with FGMRES \cite{Saad1993} preconditioned with block preconditioning techniques and the Schur complements are controlled with an augmented Lagrangian term \cite{Papadopoulos2021c}. The individual blocks are solved by a sparse LU factorization with MUMPS \cite{mumps} and PETSc \cite{petsc}. In the Taylor--Hood discretization, we invert the entire linear system directly with MUMPS. There are no known  solutions of the infinite-dimensional problem  for choices of the inverse permeability, $\alpha$, used in practice. Hence, the errors are measured with respect to the most heavily-refined finite element solution where $h = 1.41 \times 10^{-3}$ resulting in 16,389,121 degrees of freedom for the first-order BDM discretization. 

Although, for each isolated minimizer, \cref{th:FEMexistence} guarantees the existence of a converging sequence of solutions, it does not guarantee that the sequence is unique. In principle, there can be infinitely many different subsequences of finite element solutions that strongly converge to the same  minimizer of the infinite-dimensional problem  at different convergence rates. Separate subsequences can appear in oscillations in the error and cause difficulty in the interpretation of the convergence plots \cite[Sec.~5]{Papadopoulos2021b}. Here, we attempt to find solutions along the same sequence by first computing the solutions on a coarse mesh, uniformly refining the mesh, and successively interpolating the solutions onto the finer mesh as initial guesses for the deflated barrier method.

We now describe the double-pipe problem of Borrvall and Petersson \cite[Sec.~4.5]{Borrvall2003} with a modification to the boundary conditions that improves the guaranteed regularity of the solutions. The design domain is a rectangle $\Omega = (0,3/2) \times (0,1)$ with two inlets and two outlets. The volume fraction is chosen to be $\gamma = 1/3$ and the inverse permeability $\alpha$ is given by \cref{eq:alphachoice} with $q=1/10$ and $\bar \alpha = 2.5 \times 10^4$. The boundary conditions are given by:
\begin{align}
\label{eq:doublepipebcs}
 \vect{g}(x,y) = 
\begin{cases}
\left(\exp(1-\frac{1}{1-(12y-9)^2}), 0\right)^\top & \text{if} \;\; 2/3 \leq y \leq 5/6,\, x = 0 \; \text{or} \; 3/2,\;\; \\
\left(\exp(1-\frac{1}{1-(12y-3)^2}), 0\right)^\top & \text{if} \;\; 1/6 \leq y \leq 1/3,\, x = 0 \; \text{or} \; 3/2,\;\; \\
(0,0)^\top & \text{elsewhere on} \;\partial \Omega.
\end{cases}
\end{align}
This problem supports  exactly  two isolated minimizers in form of a straight channels solution and double-ended wrench solution as shown in \cref{fig:double-pipe-solns}. 
\begin{figure}[ht]
\centering
\includegraphics[width = 0.3\textwidth]{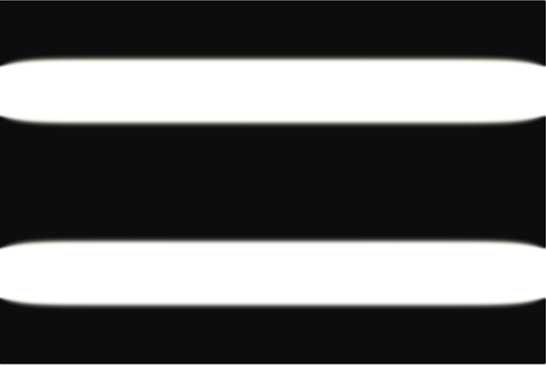}
\includegraphics[width = 0.3\textwidth]{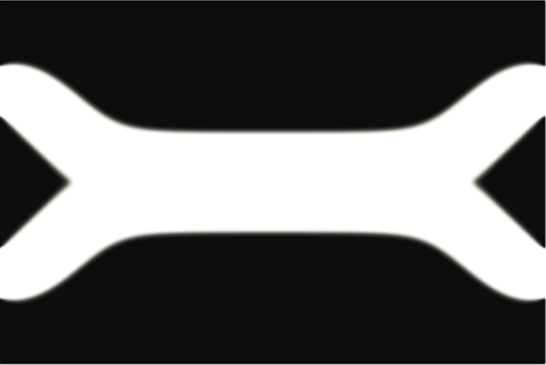}
\caption{The material distribution of the straight channels (left) and double-ended wrench (right) solutions of the double-pipe optimization problem. The black regions represent where $\rho = 0$ and the white regions are where $\rho =1$. The mesh size is $h = 1.41 \times 10^{-3}$ and a first-order BDM discretization for the velocity-pressure and a discontinuous piecewise constant discretization for the material distribution are used. This results in 16,389,121 degrees of freedom.} 
\label{fig:double-pipe-solns}
\end{figure}

It can be checked that the boundary datum $\vect{g}$ can be expressed as the trace of a function $\hat{\vect{g}} \in H^2(\Omega)^d$. Hence, since the domain is convex and the forcing term is smooth, we have that, by regularity results for isolated minimizers of the Borrvall--Petersson problem \cite[Lem.~5]{Papadopoulos2021b}, $(\vect{u}, p) \in H^2(\Omega)^d \times H^1(\Omega)$ for both minimizers. Hence, the trace of $\nabla \vect{u}$ is well-defined on the faces of each element and the consistency result in \cref{prop:consistencyah} holds. All the conditions of \cref{th:FEMexistence} are satisfied and hence there exists a sequence of solutions to \cref{FOC1h}--\cref{FOC3h} that converges strongly to the straight channel solution and a different sequence of solutions that converges to the double-ended wrench. The existence of these sequences are numerically verified in \cref{fig:double-pipe-bdm-convergence} for a  $\mathrm{DG}_0 \times \mathrm{BDM}_1 \times \mathrm{DG}_0$ discretization for $(\rho_h, \vect{u}_h, p_h)$. 
\begin{figure}[ht!]
\centering
\includegraphics[width = 0.49\textwidth]{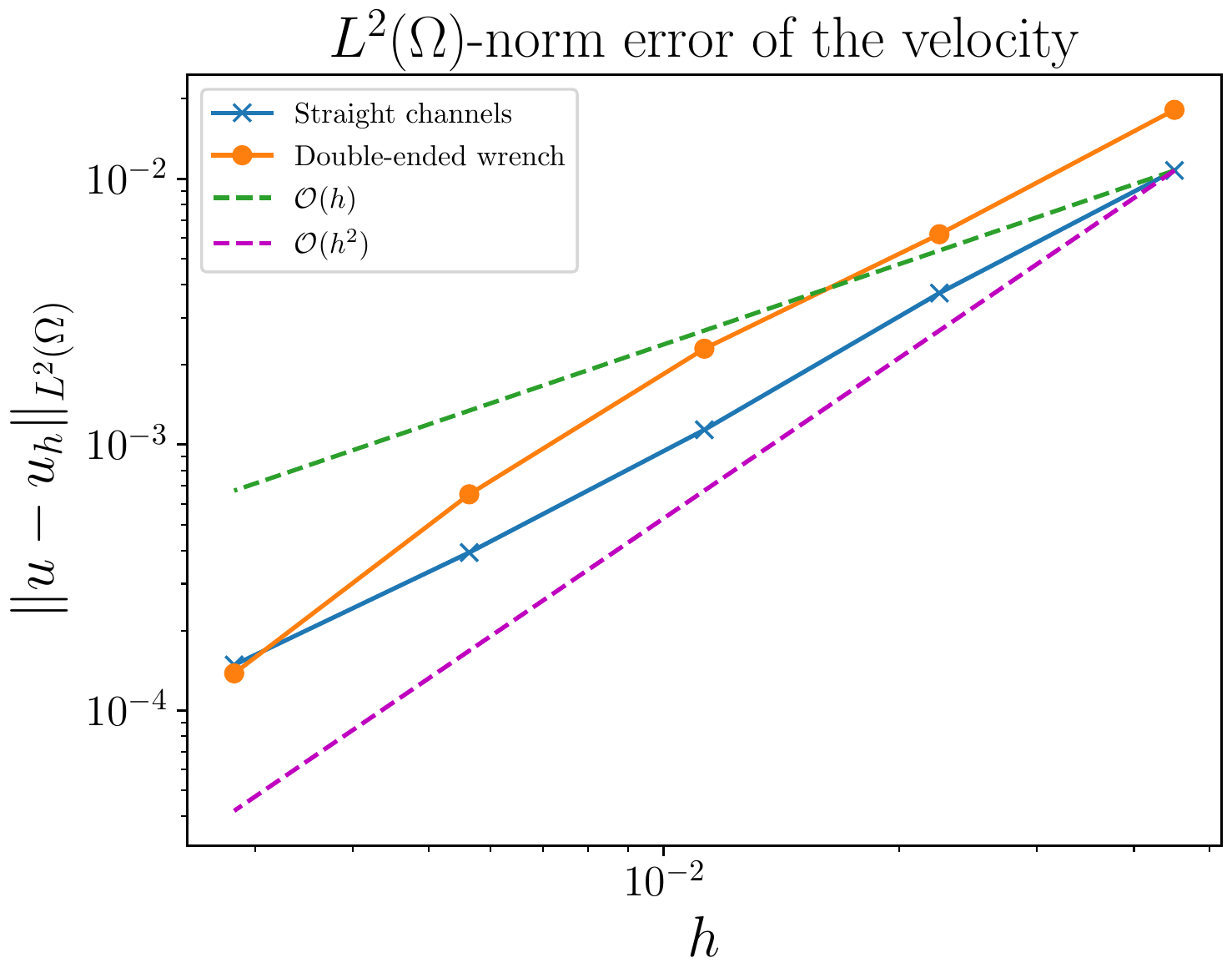}
\includegraphics[width = 0.49\textwidth]{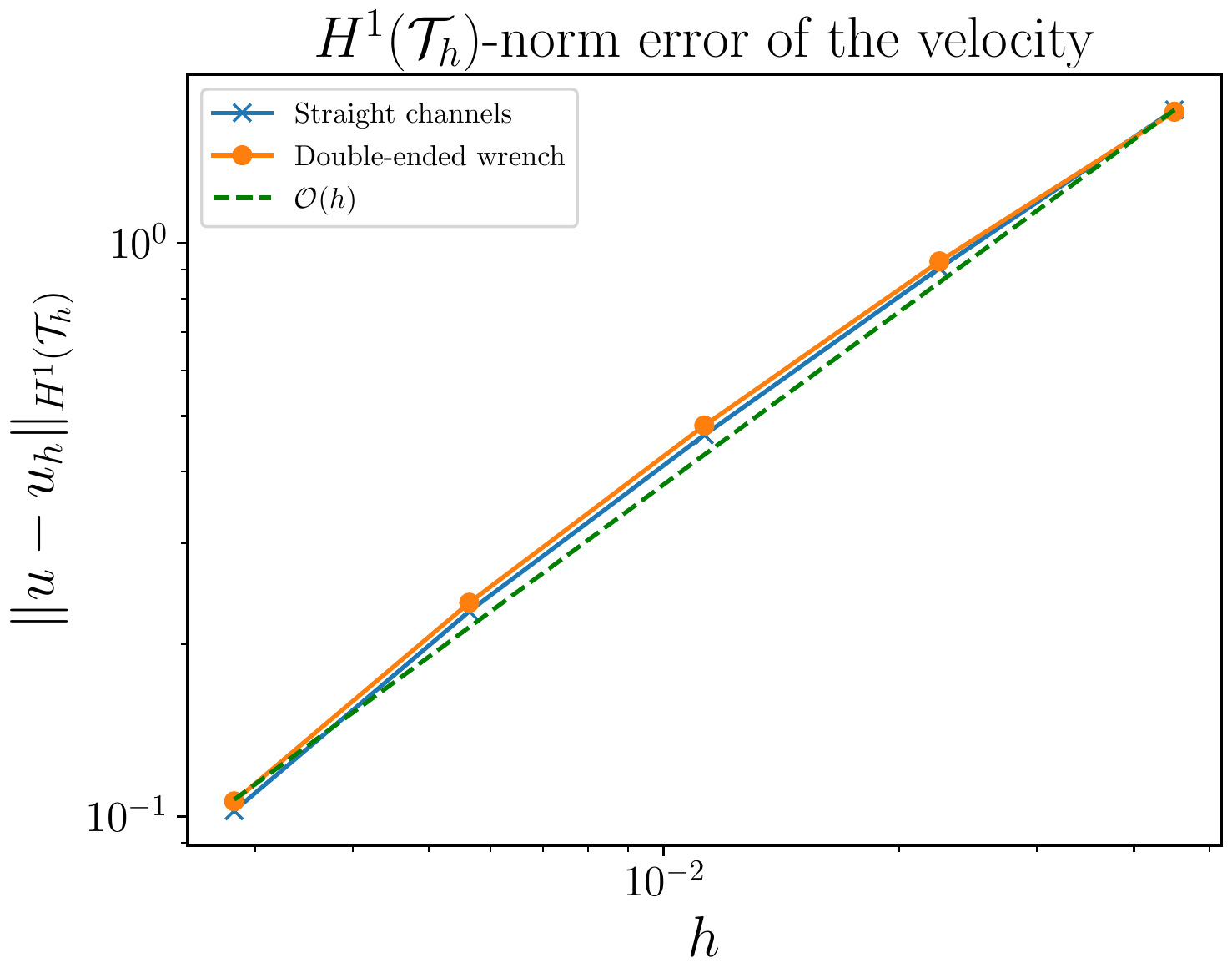}
\includegraphics[width = 0.49\textwidth]{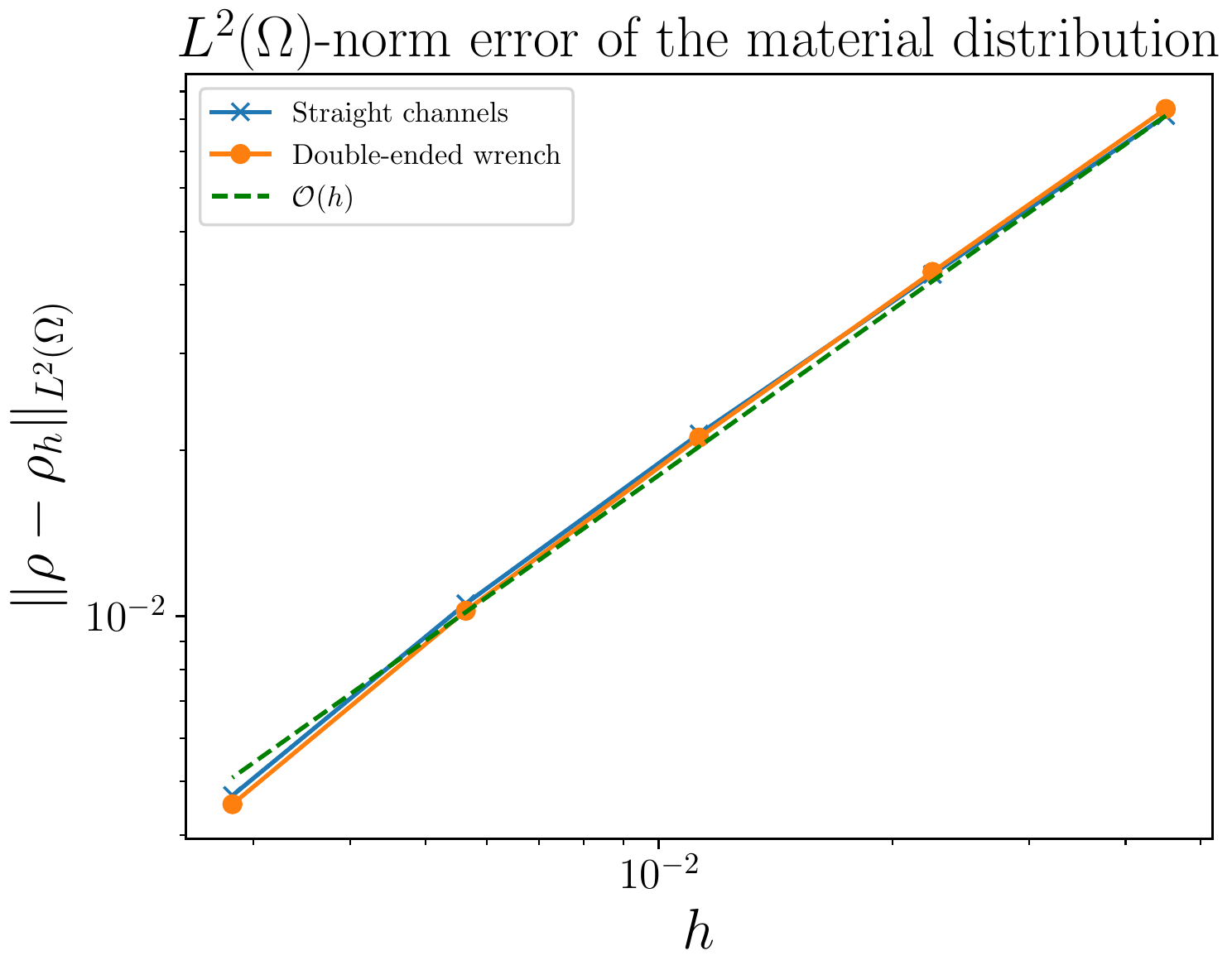}
\includegraphics[width = 0.49\textwidth]{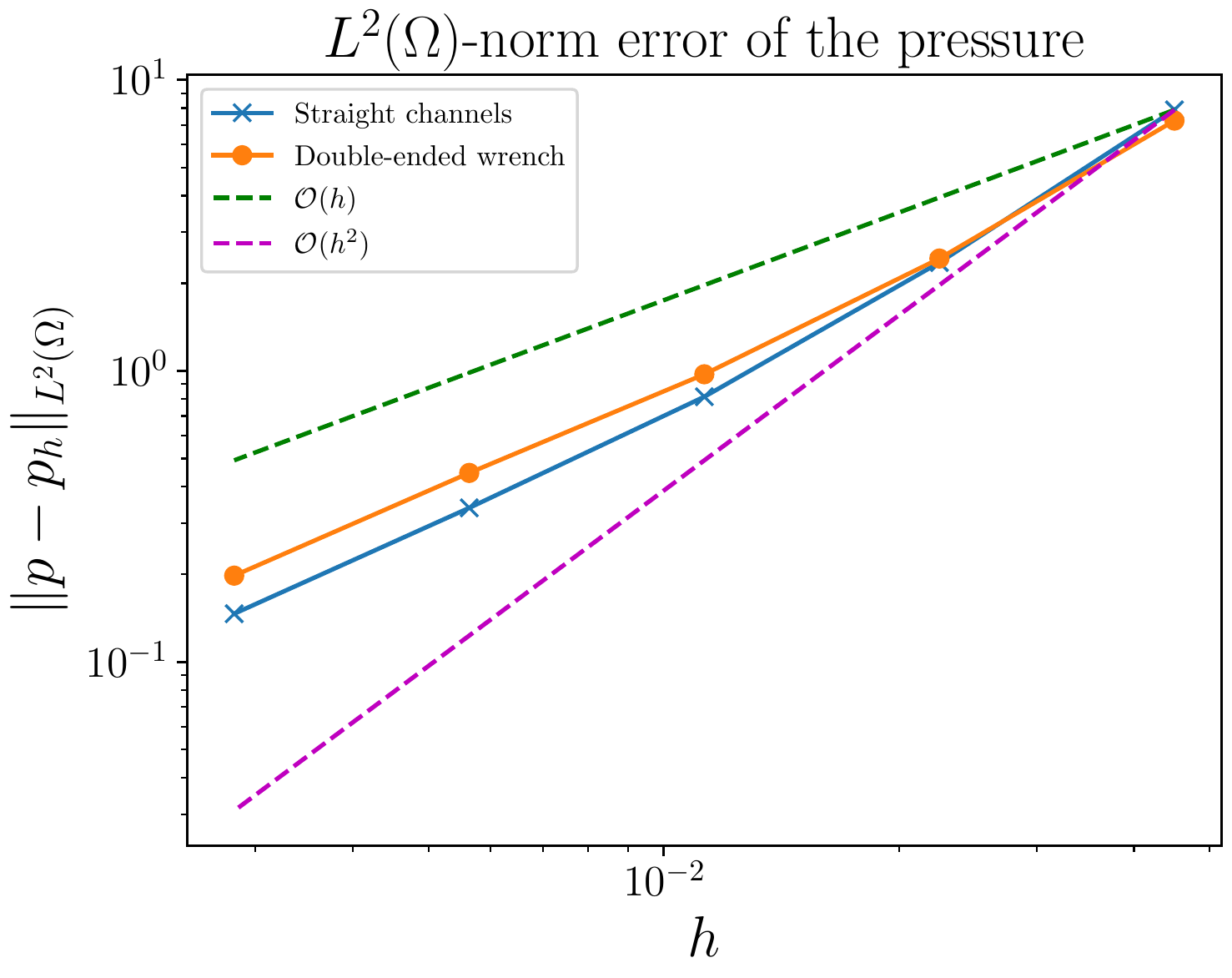}
\caption{The convergence of $\vect{u}_h$, $\rho_h$, and $p_h$ for the double-pipe problem for both the straight channels and double-ended wrench solutions on a sequence of uniformly refined meshes with a $\mathrm{DG}_0 \times \mathrm{BDM}_1 \times \mathrm{DG}_0$ discretization for $(\rho_h, \vect{u}_h, p_h)$.} 
\label{fig:double-pipe-bdm-convergence}
\end{figure}

We report the values of $\| \divv{\vect{u}_h} \|_{L^2(\Omega)}$ in \cref{tab:div-violation} for the BDM discretization alongside the equivalent solutions computed with a Taylor--Hood  $(\mathrm{CG}_2)^2 \times \mathrm{CG}_1$ discretization for the velocity-pressure pair and a $\mathrm{DG}_0$ discretization for the material distribution on the same meshes. Even on coarse meshes, the $L^2$-norm of the divergence of the velocity in the BDM discretization is small with values in the range of $10^{-6}\sim 10^{-9}$ for both minimizers.  Many of the BDM discretization values are roughly the square root of \texttt{Float64} machine precision, denoted $\mathrm{\texttt{eps}}$. Suppose that \cref{FOC2h} is satisfied up to machine precision, then we have that $|b(\vect{u}_h, q_h)| \leq \mathrm{\texttt{eps}}$. Now, by choosing $q_h = \divv{\vect{u}_h}$, we note that $\| \divv{\vect{u}_h} \|_{L^2(\Omega)} \leq \sqrt{\mathrm{\texttt{eps}}}$.  By contrast, the pointwise violation of the incompressibility constraint for the Taylor--Hood discretization manifests as relatively large values of $\| \divv{\vect{u}_h} \|_{L^2(\Omega)}$. Even on the finest mesh, where $h = 2.82 \times 10^{-3}$ resulting in 4,512,004 degrees of freedom, the $L^2$-norm is still $\mathcal{O}(10^{-3})$, 4 orders of magnitude larger than the equivalent BDM discretization. 

\begin{table}[ht]
\small
\centering
\begin{tabular}{|l|ll|ll|}
\hline
& \multicolumn{2}{c|}{Straight channels} & \multicolumn{2}{c|}{Double-ended wrench} \\
\hline
$h$ & BDM & Taylor--Hood & BDM & Taylor--Hood  \\ \hline
$ 4.51 \times 10^{-2}$ & $1.00 \times 10^{-8}$ & $2.49 \times 10^{-1}$  & $2.69 \times 10^{-6}$  & $3.25 \times 10^{-1}$ \\
$ 2.25 \times 10^{-2}$ & $6.35 \times 10^{-9}$ & $1.09 \times 10^{-1}$  & $2.75 \times 10^{-8}$  & $1.35 \times 10^{-1}$ \\
$ 1.13 \times 10^{-2}$ & $1.59 \times 10^{-7}$ & $3.95 \times 10^{-2}$  & $2.62 \times 10^{-8}$  & $4.66 \times 10^{-2}$ \\
$ 5.63 \times 10^{-3}$ & $4.19 \times 10^{-8}$ & $1.19 \times 10^{-2}$  & $1.48 \times 10^{-7}$  & $1.36\times 10^{-2}$ \\
$ 2.82 \times 10^{-3}$ & $4.97 \times 10^{-7}$ & $3.17 \times 10^{-3}$  & $2.98 \times 10^{-7}$  & $3.58\times 10^{-3}$ \\
\hline
\end{tabular}
\caption{Reported values for  $\| \divv{\vect{u}_h} \|_{L^2(\Omega)}$ in a BDM and Taylor--Hood discretization for the double-pipe problem as measured on five meshes in a uniformly refined mesh hierarchy. } 
\label{tab:div-violation}
\end{table}

\textbf{Code availability:} For reproducibility, the implementation of the deflated barrier method used in this work, as well as scripts to generate the convergence plots and solutions can be found at \url{https://github.com/ioannisPApapadopoulos/fir3dab/}. The version of the software used in this paper is archived on Zenodo \cite{dbmcode2021b}.

\section{Conclusions}
In this work we studied the convergence of a divergence-free discontinuous Galerkin finite element discretization of the fluid topology optimization model of Borrvall and Petersson \cite{Borrvall2003}. Our approach extends the techniques used by Papadopoulos and S\"uli \cite{Papadopoulos2021b} for $H^1$-conforming finite element approximations of the velocity. The nonconvexity of the optimization problem was handled by fixing any isolated minimizer and introducing a modified optimization problem with the chosen isolated minimizer as its unique solution. We then showed that there exists a sequence of discretized solutions that converges to the minimizer in the appropriate norms. In particular, $\| \vect{u} - \vect{u}_h\|_{H^1(\mathcal{T}_h)} \to 0$, $\|\rho - \rho_h\|_{L^s(\Omega)} \to 0$, $s \in [1,\infty)$, and $\|p-p_h\|_{L^2(\Omega)} \to 0$. The modified optimization problem was related back to the original optimization problem by showing that a subsequence of the strongly converging minimizers also satisfy the first-order optimality conditions of the original problem. Moreover, these first-order optimality conditions can be solved numerically. Finally, we numerically verified that these sequences exist and compared the improvement in the $L^2$-norm of the incompressibility constraint in the discretized solutions.  Future work could include adapting these methods to other topology optimization problems including extensions of the Borrvall--Petersson problem with a more sophisticated fluid flow, as well as topology optimization formulations for cantilevers and MBB beams that utilize linear elasticity.

\appendix
\section{Proof of \cref{prop:FEMconvergence1.5}}
\label{app:prop1.5}
We first quote a couple of propositions that are required in the proof of \cref{prop:FEMconvergence1.5}. The proof of the following proposition concerning the support of $\rho$ can be found in Papadopoulos and S\"uli \cite[Prop.~3]{Papadopoulos2021b}.
\begin{proposition}[Support of $\rho$]
\label{prop:rhosupport}
Suppose that $\Omega \subset \mathbb{R}^d$ is a Lipschitz domain,  $d \in \{2,3\}$, and $\alpha$ satisfies properties \ref{alpha1}--\ref{alpha4}. Further assume that the minimizer $(\vect{u},\rho) \in H^1_{\vect{g},\mathrm{div}}(\Omega)^d \times C_\gamma$ of (\ref{borrvallmin}) is  a strict minimizer. Then, $\mathrm{supp}(\rho) \subseteq U$, where $U\coloneqq\mathrm{supp}(\vect{u})$.
\end{proposition}
\begin{remark}
\cref{prop:rhosupport} is proved by contradiction. Suppose that there exists a strict minimizer $(\vect{u},\rho)$ such that $\mathrm{supp}(\rho) \not\subseteq U$. Then, it is possible to construct a pair $(\vect{u},\tilde{\rho})$ arbitrarily close to $(\vect{u},\rho)$ such that $J(\vect{u},\rho) = J(\vect{u},\tilde{\rho})$. Hence, $(\vect{u},\rho)$ cannot be a strict minimizer. 
\end{remark}

The following proposition concerning strong convergence of $\rho_h$ in sets where $\rho = 0$ or $\rho =1 $ a.e.~is thanks to Petersson \cite[Cor.~3.2]{Petersson1999} and can be found, as stated, in Papadopoulos and S\"uli \cite[Cor.~1]{Papadopoulos2021b}.
\begin{proposition}[Strong convergence of $\rho_h$ in $L^s(\Omega_b)$]
\label{prop:Omegab}
Fix an isolated minimizer $(\vect{u},\rho)$ of \cref{borrvallmin} and suppose that the conditions of \cref{th:FEMexistence} hold. Let $\Omega_b$ be any measurable subset of $\Omega$ of positive measure on which $\rho$ is equal to zero or one a.e.~(if such a set exists). Suppose that there exists a sequence of finite element minimizers, $\rho_h$, of \cref{BPh} such that $\rho_h \weakstar \rho$ weakly-* in $L^\infty(\Omega)$. Then, $\rho_h \to \rho$ strongly in $L^s(\Omega_b)$, where $s \in [1,\infty)$.
\end{proposition}
\begin{proof}
If $\rho_h \weakstar \rho$ weakly-* in $L^\infty(\Omega)$, then by definition, $\int_\Omega \rho_h \phi \, \dx \to \int_\Omega \rho \phi \, \dx$ for all $\phi \in L^1(\Omega)$. Consider $\Omega_b = \Omega_0 \cup \Omega_1$ where $\Omega_0 \coloneqq \{ \rho = 0 \}$ and $\Omega_1 \coloneqq \{ \rho = 1\}$. Suppose that $|\Omega_0|>0$. Choose $\phi = \chi_{\Omega_0} \in L^1(\Omega)$ (the characteristic function for $\Omega_0$). Then,
\begin{align}
\int_{\Omega_0} \rho_h \, \dx \to \int_{\Omega_0} \rho \, \dx = 0,
\end{align}
as $\rho = 0$ a.e.~in $\Omega_0$. Hence, $\int_{\Omega_0} |\rho_h - \rho| \dx \to 0$ as $0 \leq \rho_h$. Similarly, if $|\Omega_1|>0$, by choosing $\phi = \chi_{\Omega_1} \in L^1(\Omega)$ and utilizing that $\rho_h \leq 1$, we find that $\int_{\Omega_1} |\rho - \rho_h| \dx \to 0$. Therefore, $\| \rho - \rho_h \|_{L^1(\Omega_b)} \to 0$. Consider any $s \in (1,\infty)$. Then,
\begin{align}
\int_{\Omega_b} |\rho - \rho_h|^s \dx \leq \int_{\Omega_b} |\rho - \rho_h|^{s-1}|\rho - \rho_h| \dx \leq \| \rho - \rho_h\|_{L^1(\Omega_b)} \to 0,
\end{align}
where the second inequality holds since $0 \leq \rho, \rho_h \leq 1$. Therefore, $\| \rho - \rho_h \|_{L^s(\Omega_b)} \to 0$ for any $s \in [1,\infty)$. 
\end{proof}

We now reproduce the proof of \cref{prop:FEMconvergence1.5} as found in \cite[Prop.~5]{Papadopoulos2021b}, with some small modifications. 
\begin{proof}[Proof of \cref{prop:FEMconvergence1.5}]
We note that $C_{\gamma,h} \cap B_{r/2,L^2(\Omega)}(\rho)$ is a convex set, and hence for any $\eta_h \in C_{\gamma,h} \cap B_{r/2,L^2(\Omega)}(\rho)$, $t \in [0,1]$, we have that $ \rho_h  + t(\eta_h - \rho_h) \in C_{\gamma,h} \cap B_{r/2,L^2(\Omega)}(\rho)$. Since $(\vect{u}_h, \rho_h)$ is a global minimizer of \cref{BPh}, we note that
\begin{align}
\frac{1}{t} \left[ J_h(\vect{u}_h,  \rho_h  + t(\eta_h - \rho_h)) - J_h(\vect{u}_h, \rho_h) \right] \geq 0.
\end{align}
By taking the limit $t \to 0$ and noting that by assumption \labelcref{alpha4}, $\alpha$ is continuously differentiable, we deduce that
\begin{align}
\int_\Omega \alpha'(\rho_h) |\vect{u}_h|^2 (\eta_h - \rho_h) \dx \geq 0 \;\; \text{for all} \;\; \eta_h \in  C_{\gamma,h} \cap B_{r/2,L^2(\Omega)}(\rho). \label{ch3:femconvergence1.5-1}
\end{align}
Hence, \cref{FOC3} and \cref{ch3:femconvergence1.5-1} imply that for all $\eta \in C_\gamma$ and $\eta_h \in  C_{\gamma,h} \cap B_{r/2,L^2(\Omega)}(\rho)$ we have that
\begin{align}
\int_\Omega \alpha'(\rho) |\vect{u}|^2 \rho \; \dx &\leq \int_\Omega \alpha'(\rho) |\vect{u}|^2 \eta \; \dx, \label{ch3:materialerror1}\\
\int_\Omega \alpha'(\rho_h) |\vect{u}_h|^2 \rho_h \; \dx &\leq \int_\Omega \alpha'(\rho_h) |\vect{u}_h|^2 \eta_h \; \dx. \;\;\label{ch3:materialerror2}
\end{align}
By subtracting $\int_\Omega \alpha'(\rho) |\vect{u}|^2 \rho_h \dx$ from \cref{ch3:materialerror1} and $\int_\Omega \alpha'(\rho_h) |\vect{u}_h|^2 \rho \, \dx$ from \cref{ch3:materialerror2}, we see that
\begin{align}
\int_\Omega \alpha'(\rho) |\vect{u}|^2 (\rho - \rho_h) \; \dx &\leq \int_\Omega \alpha'(\rho)  |\vect{u}|^2 (\eta - \rho_h)\dx, \label{ch3:materialerror3}\\
\int_\Omega \alpha'(\rho_h)  |\vect{u}_h|^2 (\rho_h - \rho)\; \dx &\leq \int_\Omega \alpha'(\rho_h)  |\vect{u}_h|^2 (\eta_h - \rho) \; \dx . \label{ch3:materialerror4}
\end{align}
Summing \cref{ch3:materialerror3} and \cref{ch3:materialerror4} and rearranging the left-hand side, we see that
\begin{align}
\begin{split}
&\int_\Omega (\alpha'(\rho) - \alpha'(\rho_h))|\vect{u}|^2 (\rho - \rho_h)\dx + \int_\Omega \alpha'(\rho_h) (|\vect{u}|^2 - |\vect{u}_h|^2)(\rho - \rho_h) \dx\\
&\indent \leq \int_\Omega \alpha'(\rho)  |\vect{u}|^2 (\eta - \rho_h)\dx+ \int_\Omega \alpha'(\rho_h) |\vect{u}_h|^2 (\eta_h - \rho) \dx. \label{ch3:materialerror5}
\end{split}
\end{align}
By fixing $\eta = \rho_h \in C_\gamma$ and subtracting the second term on the left-hand side of \cref{ch3:materialerror5} from both sides we deduce that
\begin{align}
\label{ch3:materialerror6}
\begin{split}
&\int_\Omega (\alpha'(\rho) - \alpha'(\rho_h))|\vect{u}|^2 (\rho - \rho_h)\dx\\
&\indent \leq \int_\Omega \alpha'(\rho_h) |\vect{u}_h|^2 (\eta_h - \rho) \dx 
 + \int_\Omega \alpha'(\rho_h) (|\vect{u}_h|^2-|\vect{u}|^2)(\rho - \rho_h) \dx.
\end{split}
\end{align}
By an application of the mean value theorem, we note that there exists a $c\in (0,1)$ such that
\begin{align}
\begin{split}
&\int_\Omega (\alpha'(\rho) - \alpha'(\rho_h)) |\vect{u}|^2 (\rho - \rho_h)\dx\\
& \indent = \int_\Omega \alpha''(\rho_h + c(\rho - \rho_h))  |\vect{u}|^2 (\rho - \rho_h)^2\dx.
\end{split}\label{ch3:materialerror11}
\end{align}
Since by assumption \labelcref{alpha2}, $\alpha$ is strongly convex and by \labelcref{alpha4} it is twice continuously differentiable, there exists a constant $\alpha''_{\mathrm{min}} > 0$ such that
\begin{align}
\alpha''_{\mathrm{min}} \leq \alpha''(y) \;\; \text{for all} \;\; y \in [0,1]. \label{alphabound}
\end{align}
Therefore by \cref{alphabound} and the definition of $U_\theta$ (given in \cref{th:FEMexistence}) we bound \cref{ch3:materialerror11} from below: 
\begin{align}
\label{ch3:materialerror7}
\begin{split}
 &\int_\Omega \alpha''(\rho_h + c(\rho - \rho_h)) |\vect{u}|^2 (\rho - \rho_h)^2\dx\\
 & \indent \geq  \int_{U _\theta} \alpha''(\rho_h + c(\rho - \rho_h))  |\vect{u}|^2 (\rho - \rho_h)^2\dx
\geq \alpha''_{\mathrm{min}} \theta \|\rho - \rho_h\|^2_{L^2(U_\theta)}.
\end{split}
\end{align}
Now we bound the right-hand side of \cref{ch3:materialerror6} as follows,
\begin{align}
\label{ch3:materialerror8}
\begin{split}
&\int_\Omega \alpha'(\rho_h) |\vect{u}_h|^2 (\eta_h - \rho) \dx 
 + \int_\Omega \alpha'(\rho_h) (|\vect{u}_h|^2-|\vect{u}|^2)(\rho - \rho_h) \dx \\
&\indent \leq 2\alpha'_{\text{max}}(\|\vect{u}\|^2_{L^4(\Omega)} + \|\vect{u} - \vect{u}_h\|^2_{L^4(\Omega)}) \| \rho - \eta_h \|_{L^2(\Omega)} \\
&\indent\indent+ \alpha'_{\text{max}} \|\rho-\rho_h\|_{L^q(\Omega)} \| \vect{u} + \vect{u}_h \|_{L^{q'}(\Omega)} \| \vect{u} - \vect{u}_h \|_{L^2(\Omega)},
\end{split}
\end{align}
where $2<q' < \infty$ in two dimensions, $2<q' \leq 6$ in three dimensions, and $q = 2q'/(q'-2)$. We note that 
\begin{align}
\begin{split}
&\| \vect{u} + \vect{u}_h \|_{L^{q'}(\Omega)} \leq \| \vect{u} \|_{L^{q'}(\Omega)}  + \| \vect{u}_h \|_{L^{q'}(\Omega)} \\
&\indent \leq \| \vect{u} \|_{H^1(\Omega)}  + \| \vect{u}_h \|_{H^1(\mathcal{T}_h)} \leq \hat{C} < \infty,
\end{split} \label{ch3:materialerror13}
\end{align}
where the second inequality holds thanks to the Sobolev embedding theorem and the broken Friedrichs-type inequality as found in Buffa and Ortner \cite[Cor.~4.3]{Buffa2009}.
Combining \cref{ch3:materialerror6}--\cref{ch3:materialerror13} we see that
\begin{align}
\label{ch3:materialerror9}
\|\rho - \rho_h\|^2_{L^2(U_\theta)} \leq C\left( \| \rho - \eta_h \|_{L^2(\Omega)} +  \|\rho-\rho_h\|_{L^q(\Omega)}\| \vect{u} - \vect{u}_h \|_{L^2(\Omega)} \right),
\end{align}
where $C = C(\alpha'_{\text{max}}, \alpha''_{\mathrm{min}}, \theta, \|\vect{u}\|_{L^4(\Omega)}, \hat{C})$. By assumption \labelcref{ass:dense}, there exists a sequence of finite element functions $\tilde{\rho}_h \in C_{\gamma,h}$ such that $\tilde{\rho}_h \to \rho$ strongly in $L^2(\Omega)$. Thanks to the strong convergence, we note that for sufficiently small $h$, $\tilde{\rho}_h \in C_{\gamma,h} \cap B_{r/2,L^2(\Omega)}(\rho)$. Hence we fix $\eta_h =  \tilde{\rho}_h$. By \cref{prop:FEMconvergence}, we know that $\vect{u}_h \to \vect{u}$ strongly in $L^2(\Omega)^d$ and since $\rho \in C_\gamma$, $\rho_h \in C_{\gamma,h} \subset C_{\gamma}$, then $\| \rho - \rho_h \|_{L^q(\Omega)}\leq |\Omega|^{1/q} \| \rho - \rho_h \|_{L^\infty(\Omega)} \leq |\Omega|^{1/q}$. Therefore, the right-hand side of \cref{ch3:materialerror9} tends to zero as $h \to 0$. Hence, we deduce that 
\begin{align}
\rho_h \to \rho \;\; \text{strongly in} \;\; L^2(U_\theta), \;\; \theta > 0. 
\label{ch3:materialerror15}
\end{align}
We define $U$ as $U \coloneqq \mathrm{supp}(\vect{u})$. Now we note that
\begin{align}
\label{ch3:rhoOmegaSplit}
\|\rho - \rho_h\|_{L^2(\Omega)}  = \|\rho - \rho_h\|_{L^2(U_\theta)} + \|\rho - \rho_h\|_{L^2(U\backslash U_\theta)} + \|\rho - \rho_h\|_{L^2(\Omega\backslash U)}.
\end{align}
If $U\backslash U_\theta$ or $\Omega \backslash U$ are empty, we neglect the corresponding term in \cref{ch3:rhoOmegaSplit} with no loss of generality. Suppose $\Omega \backslash U$ is non-empty.  By definition of $U$, $\vect{u} = \vect{0}$ a.e.~in $\Omega \backslash U$. \cref{prop:rhosupport} implies that $\rho = 0$ a.e.~in $\Omega \backslash U$. Hence, $\Omega \backslash U \subseteq \Omega_b$ where $\Omega_b = \{\rho=1\}\cup\{\rho=0\}$. Moreover, in \cref{prop:FEMconvergence} we showed that $\rho_h \weakstar \rho$ weakly-* in $L^\infty(\Omega)$. Therefore, \cref{prop:Omegab} implies that 
\begin{align}
\label{ch3:materialerror16}
\rho_h \to \rho \;\;\text{strongly in} \;\; L^2(\Omega \backslash U). 
\end{align}
 Suppose $U\backslash U_\theta$ is non-empty. Since, $\rho, \rho_h \in C_\gamma$ we see that 
\begin{align}
\label{ch3:materialerror17}
\|\rho - \rho_h\|_{L^2(U\backslash U_\theta)} \leq |U \backslash U_\theta|^{1/2} \to 0 \;\; \text{as} \;\; \theta \to 0.
\end{align} 
Therefore, by first taking the limit as $h \to 0$ and then by taking the limit as $\theta \to 0$, \cref{ch3:materialerror15}--\cref{ch3:materialerror17} imply that $\rho_h \to \rho$ strongly in $L^2(\Omega)$. 

Since $\| \rho - \rho_h \|_{L^1(\Omega)} \leq |\Omega|^{1/2} \|\rho - \rho_h \|_{L^2(\Omega)}$, we see that $\rho_h \to \rho$ strongly in $L^1(\Omega)$. Hence, for any $s \in [1,\infty)$,
\begin{align}
\int_{\Omega} | \rho - \rho_h|^s \dx = \int_{\Omega}  | \rho - \rho_h|^{s-1}  | \rho - \rho_h|\dx \leq 1^{s-1} \| \rho - \rho_h \|_{L^1(\Omega)},
\end{align}
which implies that $\rho_h \to \rho$ strongly in $L^s(\Omega)$.
\end{proof}

\section*{Acknowledgments}
The author would like to thank Pablo Alexei Gazca-Orozco for a useful discussion on weak compactness results, and Endre S\"uli  and Patrick Farrell for their comments on the manuscript. The author would also like to thank the anonymous reviewers for their insightful comments.

\clearpage
\bibliographystyle{siamplain}
\bibliography{references}
\end{document}